\documentclass[11pt,a4paper]{amsart}
\usepackage[utf8]{inputenc}
\usepackage[T1]{fontenc}
\usepackage{amsmath}
\usepackage{amsfonts}
\usepackage{amssymb}
\usepackage{vmargin}
\usepackage{setspace}
\usepackage{xcolor}
\definecolor{vertfonce}{rgb}{0.20, 0.46, 0.25}
\definecolor{rougefonce}{rgb}{0.64, 0.09, 0.20}
\usepackage[breaklinks=true,
            colorlinks=true,
            linkcolor=rougefonce,
            citecolor=vertfonce]{hyperref}
\usepackage{mathrsfs, enumerate, csquotes, color}
\usepackage{tikz}
\author{L\' eo Morin}
\title[Birkhoff normal form for constant-rank magnetic fields]{A Semiclassical Birkhoff Normal Form for constant-rank magnetic fields}
\address{L\' eo Morin, Univ Rennes, CNRS, IRMAR - UMR 6625, F-35000 Rennes, France}
\email{leo.morin@univ-rennes1.fr}

\usepackage[explicit]{titlesec}
\titleformat{\section}{\centering\Large\bfseries}{\thesection \ --}{0.7em}{\Large\bfseries #1}
\titleformat{\subsection}{\centering\large\bfseries}{\thesubsection \ --}{0.4em}{\large\bfseries #1}
\titleformat{\subsubsection}{\centering\bfseries}{\thesubsubsection \ --}{0.4em}{\bfseries #1}

\numberwithin{equation}{section}

\newtheorem*{theorem*}{\sffamily Theorem}
\newtheorem{theorem}{\sffamily Theorem}[section]
\newtheorem{lemma}[theorem]{\sffamily Lemma}

\newtheorem{corollary}[theorem]{\sffamily Corollary}
\newtheorem{assumption}{\sffamily Assumption}
\newtheorem{remark}[theorem]{\sffamily Remark}

\newcommand{\R}{\mathbf{R}}
\newcommand{\N}{\mathbf{N}}
\newcommand{\Z}{\mathbf{Z}}
\newcommand{\dd}{\mathrm{d}}
\newcommand{\grandO}{\mathcal{O}}
\newcommand{\Ld}{\mathsf{L}^2}
\newcommand{\Cinf}{\mathcal{C}^{\infty}}
\newcommand{\spectrum}{\mathsf{sp}}

\newcommand{\Vh}{V_{\hbar}}

\newcommand{\UhII}{U_{\mathbf{2},\hbar}}
\newcommand{\VhII}{V_{\mathbf{2},\hbar}}

\newcommand{\RhII}{\mathcal{R}_{\mathbf{2},\hbar}}

\newcommand{\A}{\mathbf{A}}
\newcommand{\B}{\mathbf{B}}

\newcommand{\ub}{\mathbf{u}}
\newcommand{\vb}{\mathbf{v}}
\newcommand{\wb}{\mathbf{w}}
\newcommand{\gb}{\mathbf{g}}
\newcommand{\fb}{\mathbf{f}}
\newcommand{\un}{\mathbf{1}}

\newcommand{\dbar}{\overline{\partial}}

\newcommand{\Ih}{\mathcal{I}_{\hbar}}
\newcommand{\Jh}{\mathcal{J}_{\hbar}}

\newcommand{\Lh}{\mathcal{L}_{\hbar}}

\newcommand{\Qh}{\mathcal{Q}_{\hbar}}
\newcommand{\Rh}{\mathcal{R}_{\hbar}}
\newcommand{\Nh}{\mathcal{N}_{\hbar}}
\newcommand{\Mh}{\mathcal{M}_{\hbar}}

\newcommand{\Uh}{\mathsf{U}_{\hbar}}

\newcommand{\taut}{\tilde{\tau}}

\newcommand{\Dom}{\mathsf{Dom}}

\newcommand{\run}{r_{\mathbf{1}}}
\newcommand{\rdeux}{r_{\mathbf{2}}}
\newcommand{\fsUn}{f^{\star}_{\un}}

\newcommand{\vect}{\mathsf{span}}
\newcommand{\ad}{\mathsf{ad}}
\newcommand{\Op}{\mathsf{Op}^w_{\hbar}}
\newcommand{\Oph}{\mathsf{Op}^w_h}
\newcommand{\Opt}{\mathsf{Op}^w_{\sharp}}

\begin{document}

\begin{abstract}
We consider the semiclassical magnetic Laplacian $\Lh$ on a Riemannian manifold, with a constant-rank and non-vanishing magnetic field $B$. Under the localization assumption that $B$ admits a unique and non-degenerate well, we construct a Birkhoff normal forms to describe the spectrum of $\Lh$ in the semiclassical limit $\hbar \rightarrow 0$. We deduce an expansion of the eigenvalues of order $\hbar$, in powers of $\hbar^{1/2}$.
\end{abstract}

\maketitle

\section{Introduction}

\subsection{Context}

We consider the semiclassical magnetic Laplacian with Dirichlet boundary conditions
\[ \Lh = (i\hbar \dd + A)^*(i\hbar \dd + A) \]
on a $d$-dimensional oriented Riemannian manifold $(M,g)$, which is either compact with boundary, or the Euclidean $\R^d$. $A$ denotes a smooth $1$-form on $M$, the magnetic potential. The magnetic field is the $2$-form $B= \dd A$. 

\medskip

The spectral theory of the magnetic Laplacian has given rise to many investigations, and appeared to have very various behaviours according to the variations of $B$ and the geometry of $M$. We refer to the books and review \cite{HelKor14, FournaisHelffer,Raymond} for a description of these works. Here we focus on the Dirichlet realisation of $\Lh$, and we give a description of semi-excited states, eigenvalues of order $\grandO(\hbar)$, in the semiclassical limit $\hbar \rightarrow 0$. As explained in the above references, the \textit{magnetic intensity} has a great influence on these eigenvalues, and one can define it in the following way.

\medskip

Using the isomorphism $T_qM \simeq T_qM^*$ given by the metric, one can define the following skew-symmetric operator $\B(q) : T_qM \rightarrow T_qM$, by:
\begin{equation}\label{defBmatrix}
 B_q(X,Y) = g_q(X, \B(q)Y), \quad \forall X,Y \in T_qM, \quad \forall q \in M\,.
 \end{equation}
The operator $\B(q)$ being skew-symmetric with respect to the scalar product $g_q$, its eigenvalues are purely imaginary and symmetric with respect to the real axis. We denote these repeated eigenvalues by
\[ \pm i \beta_j(q), \cdots, \pm i \beta_s(q), 0 \,, \]
with $\beta_j(q) > 0$. In particular, the rank of $\B(q)$ is $2s$ and may depend on $q$. However, we will focus on the constant-rank case. We denote by $k$ the dimension of the Kernel of $\B(q)$, so that $d= 2s+k$. The magnetic intensity (or "Trace+") is the following scalar-valued function,
\[ b(q) = \sum_{j=1}^s \beta_j(q) \,. \]
The function $b$ is continuous on $M$, but non-smooth in general. We are interested in discrete magnetic wells and non-vanishing magnetic fields. 

\begin{assumption}\label{assump.loc}
The magnetic intensity is non-vanishing, and admits a unique global minimum $b_0 >0$ at $q_0 \in M \setminus \partial M$, which is non-degenerate. Moreover, in the non-compact case $M= \R^d$ we assume that
\[ b_\infty := \liminf_{\vert q \vert \rightarrow + \infty} b(q) > b_0 \]
and the existence of $C >0$ such that
\[ \vert \partial_{\ell} \B_{ij}(q) \vert \leq C(1+ \vert \B(q) \vert)\,, \quad \forall \ell, i, j \quad \forall q \in \R^d \,. \]
\end{assumption}

\begin{assumption}\label{assump.constant.rank}
The rank of $\B(q)$ is constant equal to $2s >0$ on a neighborhood $\Omega$ of $q_0$.
\end{assumption}

Under Assumption \ref{assump.loc}, the following useful inequality was proven in \cite{HelMo96}. There is a $C_0 >0$ such that, for $\hbar$ small enough,
\begin{equation}\label{eq.minoration.par.b}
(1+ \hbar^{1/4}C_0) \langle \Lh u, u \rangle \geq \int_M \hbar(b(q)-\hbar^{1/4}C_0) \vert u(q) \vert^2 \dd q \,,\quad  \forall u \in \Dom (\Lh)\,.
\end{equation}

\begin{remark}
Actually, one has the better inequality obtained replacing $\hbar^{1/4}$ by $\hbar$. This was proved by Guillemin-Uribe \cite{GuiUr88} in the case of a non-degenerate $B$, by Borthwick-Uribe \cite{BorUr96} in the constant rank case, and by Ma-Marinescu \cite{MaMa02} in a more general setting.
\end{remark}

\begin{remark} Using this inequality, one can prove Agmon-like estimates for the eigenfunctions of $\Lh$. Namely, the eigenfunctions associated to an eigenvalue $< b_1 \hbar$ are exponentially small outside $K_{b_1} = \lbrace q\,, b(q) \leq b_1 \rbrace$. We will use this result to reduce our analysis to the neighborhood $\Omega$ of $q_0$. In particular, the greater $b_1$ is, the larger $\Omega$ must be.
\end{remark}

Finally, we will assume the following, in order to get smooth functions $\beta_j$ on $\Omega$.

\begin{assumption}\label{assump.betaj.simples}
$\beta_i(q_0) \neq \beta_j(q_0)$ for every $1 \leq i < j \leq s$.
\end{assumption}

Under Assumptions \ref{assump.loc} and \ref{assump.constant.rank}, estimates on the ground states of $\Lh$ in the semiclassical limit $\hbar \rightarrow 0$ were proven in several works, especially in dimension $d=2,3$.

\medskip

On $M=\R^2$, asymptotics for the $j$-th eigenvalue of $\Lh$ 
\begin{equation}\label{eq.asymp.lambdaj.2D}
\lambda_j(\Lh) = b_0 \hbar + (\alpha (2j-1) + c_1) \hbar^2 + o(\hbar^2)
\end{equation}
with explicit $\alpha, c_1 \in \R$ were proven by Helffer-Morame \cite{HelMo01} (for $j=1$) and Helffer-Kordyukov \cite{HelKor11} ($j \geq 1$). Actually, this second paper contains a description of some higher eigenvalues. They proved that, for any integers $n$, $j \in \N$, there exist $\hbar_{jn} >0$ and for $\hbar \in (0,\hbar_{jn})$ an eigenvalue $\lambda_{n,j}(\hbar) \in \mathsf{sp} ( \Lh)$ such that
\[ \lambda_{n,j}(\hbar) = (2n-1) (b_0 \hbar + ((2j-1)\alpha + c_n)\hbar^2) + o(\hbar^2) \,, \]
for another explicit constant $c_n$. In particular, it gives a description of \textit{some} semi-excited states (of order $(2n-1)b_0 \hbar$). Finally, Raymond-V\~u Ng\d{o}c \cite{Birkhoff2D} (and Helffer-Kordyukov \cite{HelKor14-2}) got a description of the whole spectrum below $b_1 \hbar$, for any fixed $b_1 \in (b_0,b_\infty)$. More precisely, they proved that this part of the spectrum is given by a familly of effective operators $\Nh^{[n]}$ ($n\in \N$) modulo $\grandO(\hbar^\infty)$. These effective operators are $\hbar$-pseudodifferential operators with principal symbol given by the function $\hbar (2n-1)b$. More interestingly, they explained why the two quantum oscillators
\[ (2n-1)b_0 \hbar\,, \quad \text{and} \quad (2j-1)\alpha \hbar^2 \,, \]
appearing in the eigenvalue asymptotics correspond to two oscillatory motions in classical dynamics : The cyclotron motion, and a rotation arround the minimum point of $b$. The results of Raymond-V\~u Ng\d{o}c were generalized to an arbitrary $d$-dimensional Riemannian manifold in \cite{MagBNF}, under the assumption $k=0$ ($\B(q)$ has full rank), proving in particular similar estimates \eqref{eq.asymp.lambdaj.2D} in a general setting. Actually, these eigenvalue estimates were proven simultaneously by Kordyukov \cite{Kor19} in the context of the Bochner Laplacian.

\medskip

In this paper we are interested on the influence of the kernel of $\B$ ($k>0$). The rank of $\B$ being even, this kernel always exists in odd dimensions : if $d=3$ the kernel directions correspond to the usual field lines. On $M=\R^3$, Helffer-Kordyukov \cite{HeKo3D} proved the existence of $\lambda_{nmj}(\hbar) \in \mathsf{sp}(\Lh)$ such that,
\begin{align*}
\lambda_{nmj}(\hbar) = & (2n-1)b_0 \hbar + (2n-1)^{1/2}(2m-1)\nu_0 \hbar^{3/2} \\
&+ ((2n-1)(2j-1)\alpha + c_{nm}) \hbar^2 + \grandO(\hbar^{9/4})\,,
\end{align*}
for some $\nu_0 >0$ and $\alpha$, $c_{nm} \in \R$. Motivated by this result and the 2D case, Helffer-Kordyukov-Raymond-V\~u Ng\d{o}c \cite{Birkhoff3D} gave a description of the whole spectrum below $b_1 \hbar$, proving in particular the eigenvalue estimates
\begin{equation}
\lambda_j(\Lh)= b_0 \hbar + \nu_0 \hbar^{3/2} + \alpha (2j-1) \hbar^2 + \grandO(\hbar^{5/2})\,.
\end{equation}
Their results exhibit a new classical oscillatory motion in the directions of the field lines, corresponding to the quantum oscillator $(2m-1)\nu_0 \hbar^{3/2}$.

\medskip

The aim of this paper is to generalize the results of \cite{Birkhoff3D} to an arbitrary Riemannian manifold $M$, under the assumptions \ref{assump.loc} and \ref{assump.constant.rank}. In particular we describe the influence of the kernel of $\B$ in a general geometric and dimensional setting. Their approach, which we adapt, is based on a \textit{semiclassical Birkhoff normal form}. The \textit{classical} Birkhoff normal form has a long story in physics, and goes back to Delaunay \cite{Delaunay} and Lindstedt \cite{Lindstedt}. This formal normal form was the starting point of a lot of studies on stability near equilibrium, and KAM theory (after Kolmogorov \cite{Kolmogorov}, Arnold \cite{Arnold}, Moser \cite{Mos62}). The works of Birkhoff \cite{Birkhoff} and Gustavson \cite{Gustavson} gave its name to this normal form. We refer to the books \cite{Mos68} and \cite{HZ94} for precise statements. Our approach here relies on a quantization. Physicists and quantum chemists already noticed in the 1980' that a quantum analogue of the Birkhoff normal form could be used to compute energies of molecules (\cite{Chemist01},\cite{Chemist02},\cite{Chemist03},\cite{Chemist04}). Joyeux and Sugny also used such techniques to describe the dynamics of excited states (see \cite{Joyeux} for example). In \cite{Sjo92}, Sjöstrand constructed a semi-classical Birkhoff normal form for a Schrödinger operator $-\hbar^2 \Delta + V$, using the Weyl quantization, to make a mathematical study of semi-excited states. In their paper \cite{Birkhoff2D}, Raymond and V\~u Ng\d{o}c had the idea to adapt this method for $\Lh$ on $\R^2$, and with Helffer and Kordyukov on $\R^3$ \cite{Birkhoff3D}. This method is reminiscent of Ivrii's approach (in his book \cite{Ivrii}).

\subsection{Main results}

The first idea is to link the classical dynamics of a particle in the magnetic field $B$ with the spectrum of $\Lh$ using pseudodifferential calculus. Indeed, $\Lh$ is a $\hbar$-pseudodifferential operator with symbol
\[H(q,p) = \vert p - A_q \vert^2 + \grandO(\hbar^2) \,, \quad \forall p \in T_qM^*\,, \forall q \in M \,,\]
and $H$ is the classical Hamiltonian associated to the magnetic field $B$. One can use this property to prove that, in the phase space $T^*M$, the eigenfunctions (with eigenvalue $< b_1 \hbar$) are microlocalized on an arbitrarily small neighborhood of
\[ \Sigma = H^{-1}(0) \cap T^*\Omega = \lbrace (q,p) \in T^* \Omega, \quad p = A_q \rbrace \,. \]
Hence, the second main idea is to find a normal form for $H$ on a neighborhood of $\Sigma$. Namely, we find canonical coordinates near $\Sigma$ in which $H$ has a "simple" form. The symplectic structure of $\Sigma$, as submanifold of $T^*M$ is thus of great interest. One can see that the restriction of the canonical symplectic form $\dd p \wedge \dd q$ on $T^*M$ to $\Sigma$ is given by $B$ (Lemma \ref{lem.Sigma.B}); and when $B$ has constant-rank, one can find Darboux coordinates $\varphi : \Omega' \subset \R^{2s+k}_{(y,\eta,t)} \rightarrow \Omega$ such that
\[ \varphi^* B = \dd \eta \wedge \dd y \,,\]
up to reducing $\Omega$. We will start from these coordinates to get the following normal form for $H$.

\begin{theorem}\label{thm.H.rond.Phi}
Under Assumptions \ref{assump.loc}, \ref{assump.constant.rank}, and \ref{assump.betaj.simples}, there exists a diffeomorphism \[ \Phi_{\mathbf{1}} : U_{\mathbf{1}}' \subset \R^{4s+2k} \rightarrow U_{\mathbf{1}} \subset T^*M \]
between neighborhoods $U_{\mathbf{1}}'$ of $0$ and $U_{\mathbf{1}}$ of $\Sigma$ such that
\[ \widehat{H}(x,\xi,y,\eta,t,\tau) := H \circ \Phi_{\mathbf{1}} (x,\xi, y,\eta,t,\tau) \]
satisfies (with the notation $\widehat{\beta}_j = \beta_j \circ \varphi$),
\[\widehat H = \langle M(y,\eta,t) \tau, \tau \rangle + \sum_{j=1}^s \widehat \beta_j(y,\eta,t) \left( \xi_j^2 + x_j^2 \right) + \grandO((x,\xi,\tau)^3) \,, \]
uniformly with respect to $(y,\eta,t)$, for some $(y,\eta,t)$-dependant positive definite matrix $M(y,\eta,t)$. Moreover, 
\[ \Phi_{\mathbf{1}}^* (\dd p \wedge \dd q) = \dd \xi \wedge \dd x + \dd \eta \wedge \dd y + \dd \tau \wedge \dd t\,. \]
\end{theorem}

\begin{remark}
We will use the following notation for our canonical coordinates: 
\[ z=(x,\xi) \in \R^{2s}\,, \quad w= (y,\eta) \in \R^{2s} \,, \quad \tau=(t,\tau) \in \R^{2k} \,.\]
This theorem gives the Tayor expansion of $H$ on a neighborhood of $\Sigma$. In particular $(x,\xi,\tau) \in \R^d$ measures the distance to $\Sigma$ whereas $(y,\eta,t) \in \R^d$ are canonical coordinates on $\Sigma$.
\end{remark}

\begin{remark}
This theorem exhibits the harmonic oscillator $\xi_j^2 + x_j^2$ in the first-order expansion of $H$. This oscillator, which is due to the non-vanishing magnetic field, corresponds to the well-known cyclotron motion. 
\end{remark}

Actually, one can use the \textit{Birkhoff normal form} algorithm to improve the remainder. Using this algorithm, we can change the $\grandO((x,\xi)^3)$ remainder into an explicit function of $\xi_j^2 + x_j^2$, plus some smaller remainders $\grandO((x,\xi)^r)$. This remainder power $r$ is restricted by resonances between the coefficients $\beta_j$. Thus, we take an integer $\run \in \N$ such that
\[ \forall \alpha \in \Z^s, \quad 0 < \vert \alpha \vert < \run \Rightarrow \sum_{j=1}^s \alpha_j \beta_j(q_0) \neq 0 \,. \]
Here, $\vert \alpha \vert = \sum_j \vert \alpha_j \vert$. Moreover, we can use the pseudodifferential calculus to apply the Birkhoff algorithm to $\Lh$, changing the classical oscillator $\xi_j^2 + x_j^2$ into the quantum harmonic oscillator
\[ \Ih^{(j)} = -\hbar^2 \partial_{x_j}^2 + x_j^2 \,,\]
whose spectrum consists of the simple eigenvalues $(2n-1)\hbar$, $n \in \N$. Following this idea we prove the following theorem.

\begin{theorem}\label{thm.main.first.bnf}
Let $\varepsilon >0$. Under Assumptions \ref{assump.loc}, \ref{assump.constant.rank} and \ref{assump.betaj.simples}, there exist $b_1 \in (b_0,b_\infty)$, an integer $N_{\text{max}}>0$ and a compactly supported function $f_{\mathbf{1}}^\star \in \Cinf( \R^{2s + 2k} \times \R^s \times [0,1))$ such that
\[ \vert f_{\mathbf{1}}^\star (y,\eta,t,\tau,I,\hbar) \vert \lesssim \left( (\vert I \vert + \hbar)^2 + \vert \tau \vert ( \vert I \vert + \hbar) + \vert \tau \vert^3 \right) \,, \]
satisfying the following properties. For $n \in \N^s$, denote by $\Nh^{[n]}$ the $\hbar$-pseudodifferential operator in $(y,t)$ with symbol
\[ N_\hbar^{[n]} = \langle M(y,\eta,t) \tau, \tau \rangle + \sum_{j=1}^s \widehat{\beta}_j(y,\eta,t) (2n_j-1)\hbar + f_{\mathbf{1}}^\star(y,\eta,t,\tau,(2n-1)\hbar,\hbar) \,.\]
For $\hbar <<1$, there exists a bijection
\[ \Lambda_\hbar : \mathsf{sp}(\Lh) \cap (-\infty, b_1 \hbar) \rightarrow \bigcup_{\vert n \vert \leq N_{\text{max}}} \mathsf{sp} \big{(} \Nh^{[n]} \big{)} \cap (-\infty, b_1 \hbar) \,, \]
such that $\Lambda_\hbar(\lambda) = \lambda + \grandO(\hbar^{\frac{\run}{2}-\varepsilon})$ uniformly with respect to $\lambda$. 
\end{theorem}

\begin{remark}
In this theorem $\mathsf{sp}(\mathcal{A})$ denotes the \textit{repeated} eigenvalues of an operator $\mathcal{A}$, so that there might be some multiple eigenvalues, but $\Lambda_\hbar$ preserves this multiplicity. We only consider self-adjoint operators with discrete spectrum.
\end{remark}

\begin{remark}\label{remark.b1.petit}
One should care of how large $b_1$ can be. As mentionned above, the eigenfunctions of energy $< b_1 \hbar$ are exponentially small outside $K_{b_1} = \lbrace q \in M\,, \quad b(q) \leq b_1 \rbrace$. Thus, we will chose $b_1$ such that $K_{b_1} \subset \Omega$, where $\Omega$ is some neighborhood of $q_0$. Hence the larger $\Omega$ is, the greater $b_1$ can be. However, there are three restrictions on the size of $\Omega$:
\begin{enumerate}
\item[•] The rank of $\B(q)$ is constant on $\Omega$,
\item[•] There exist canonical coordinates $\varphi$ on $\Omega$ (i.e. such that $\varphi^* B = \dd \eta \wedge \dd y$),
\item[•] There is no resonance in $\Omega$:
\[ \forall q \in \Omega \,, \quad \forall \alpha \in \Z^s \,, \quad 0<\vert \alpha \vert< \run \Rightarrow \sum_{j=1}^s \alpha_j \beta_j(q) \neq 0 \,. \]
\end{enumerate}
\end{remark}

\begin{remark}
If $k=0$ we recover the result of \cite{MagBNF}. Here we want to study the influence of a non-zero kernel $k >0$. This result generalizes the result of \cite{Birkhoff3D}, which corresponds to $d=3$, $s=k=1$, on the Euclidean $\R^3$. However, this generalization is not straightforward since the magnetic geometry is much more complicated in higher dimensions, in particular if $k>1$. Moreover, there is a new phenomena in higher dimensions : resonances between the functions $\beta_j$ (as in \cite{MagBNF}).
\end{remark}

The spectrum of $\Lh$ in $(-\infty,b_1\hbar)$ is reduced to the operators $\Nh^{[n]}$. Actually if we chose $b_1$ small enough, it is reduced to the first operator $\Nh^{[1]}$ (Here we denote the multi-integer $1 = (1, \cdots, 1) \in \N^s$). Hence in the second part of this paper, we study the spectrum $\Nh^{[1]}$ using a second Birkhoff normal form. Indeed, the symbol of $\Nh^{[1]}$ is
\[ N_{\hbar}^{[1]}(w,t,\tau) = \langle M(w,t) \tau, \tau \rangle + \hbar \widehat{b}(w,t) + \grandO(\hbar^2) + \grandO( \tau \hbar) + \grandO(\tau^3)\,, \]
so if we denote by $s(w)$ the minimum point of $t \mapsto \widehat b (w,t)$ (which is unique on a neighborhood of $0$), we get the following expansion
\[ N_{\hbar}^{[1]}(w,t,\tau) = \langle M(w,s(w)) \tau, \tau \rangle + \frac{\hbar}{2} \langle \frac{\partial^{2}\widehat{b}}{\partial t^{2}}(w,s(w)) \cdot (t-s(w)), t-s(w) \rangle + \cdots \]
and the principal part is a harmonic oscillator with frequences $\sqrt{\hbar} \nu_j(w)$ ($1 \leq j \leq k$) where $(\nu_j^2(w))_{1 \leq j \leq k}$ are the eigenvalues of the symmetric matrix:
\[ M(w,s(w))^{1/2} \cdot \frac{1}{2} \partial_t^2 \hat{b}(w,s(w)) \cdot M(w,s(w))^{1/2} \,.\]
These frequences are smooth non-vanishing functions of $w$ on a neighborhood of $0$, as soon as we assume that they are simple.

\begin{assumption}\label{assump.nu.j.simples}
$\nu_i(0) \neq \nu_j(0)$ for indices $1 \leq i < j \leq k$.
\end{assumption}

We fix an integer $\rdeux \in \N$ such that
\[ \forall \alpha \in \Z^k\,, \quad 0 < \vert \alpha \vert < \rdeux \Rightarrow \sum_{j=1}^k \alpha_j \nu_j(0) \neq 0 \,,\]
and we prove the following reduction theorem for $\Nh^{[1]}$.

\begin{theorem}\label{Chap4-Thm-spec-Nh0}
Let $c>0$ and $\delta \in (0,\frac{1}{2})$. Under assumptions \ref{assump.loc}, \ref{assump.constant.rank}, \ref{assump.betaj.simples} and \ref{assump.nu.j.simples}, with $k>0$, there exists a compactly supported function $f_{\mathbf{2}}^\star \in \Cinf( \R^{2s} \times \R^k \times [0,1) )$ such that
\[ \vert f_{\mathbf{2}}^\star(y,\eta,J, \sqrt{\hbar}) \vert \lesssim \left( \vert J \vert + \sqrt \hbar \right)^2 \,,\]
satisfying the following properties. For $n \in \N^k$, denote by $\Mh^{[n]}$ the $\hbar$-pseudodifferential operator in $y$ with symbol
\[ M_\hbar^{[n]}(y,\eta) = \widehat b (y,\eta,s(y,\eta)) + \sqrt{\hbar} \sum_{j=1}^k \nu_j(y,\eta)(2n_j -1) + f_{\mathbf{2}}^\star (y,\eta, (2n-1)\sqrt \hbar, \sqrt{\hbar})\,. \]
For $\hbar <<1$, there exists a bijection
\[ \Lambda_\hbar : \mathsf{sp}(\Nh^{[1]}) \cap (-\infty, (b_0 + c \hbar^\delta)\hbar) \rightarrow \bigcup_{n \in \N^k} \mathsf{sp}(\hbar \Mh^{[n]}) \cap (-\infty, (b_0 + c \hbar^\delta)\hbar)\,, \]
such that $\Lambda_\hbar(\lambda) = \lambda + \grandO( \hbar^{1+ \delta \rdeux/2})$ uniformly with respect to $\lambda$.
\end{theorem}

\begin{remark}
The threshold $b_0 + c \hbar^\delta$ is needed to get microlocalization of the eigenfunctions of $\Nh^{[1]}$ in an arbitrarily small neighborhood of $\tau =0$.
\end{remark}

\begin{remark}
This second harmonic oscillator (in variables $(t,\tau)$) corresponds to a classical oscillation in the directions of the field lines. We see that this new motion, due to the kernel of $\B$, induces powers of $\sqrt{\hbar}$ in the spectrum.
\end{remark}

As a corollary, we get a description of the low-lying eigenvalues of $\Lh$ by the effective operator $\hbar \Mh^{[1]}$.

\begin{corollary}\label{cor.reduction.Mh0}
Let $\varepsilon >0$ and $c \in (0, \min_j \nu_j(0))$. Denote by $\nu(0) = \sum_j \nu_j(0)$ and $r= \min (2 \run, \rdeux + 4)$. Under assumptions \ref{assump.loc}, \ref{assump.constant.rank}, \ref{assump.betaj.simples} and \ref{assump.nu.j.simples}, with $k>0$, there exists a bijection
\[ \Lambda_\hbar : \mathsf{sp}(\Lh) \cap (-\infty, \hbar b_0 + \hbar^{3/2}(\nu(0) + 2c)) \rightarrow \mathsf{sp}(\hbar \Mh^{[1]}) \cap (-\infty, \hbar b_0 + \hbar^{3/2}(\nu(0)+2c)) \]
such that $\Lambda_\hbar(\lambda) = \lambda + \grandO( \hbar^{r/4 - \varepsilon})$ uniformly with respect to $\lambda$.
\end{corollary}

We deduce the following eigenvalue asymptotics.

\begin{corollary}\label{cor.eigenvalue.asymptotics}
Under the assumptions of corollary \ref{cor.reduction.Mh0}, for $j \in \N$, the $j$-th eigenvalue of $\Lh$ admits an expansion
\[ \lambda_j(\Lh) = \hbar \sum_{\ell = 0}^{\lfloor r/2 \rfloor -2} \alpha_{j\ell} \hbar^{\ell/2} + \grandO( \hbar^{r/4 - \varepsilon} ) \,, \] with coefficients $\alpha_{j \ell} \in \R$ such that:
\[ \alpha_{j,0} = b_0, \quad \alpha_{j,1} = \sum_{j=1}^{k} \nu_j(0), \quad \alpha_{j,2} = E_j + c_0 \,, \]
where $c_0 \in \R$ and $\hbar E_j$ is the $j$-th eigenvalue of a $s$-dimensional harmonic oscillator.
\end{corollary}

\begin{remark}
$\hbar E_j$ is the $j$-th eigenvalue of a harmonic oscillator whose symbol is given by the Hessian at $w=0$ of $\hat b (w, s(w))$. Hence, it corresponds to a third classical oscillatory motion : a rotation in the space of field lines.
\end{remark}

\begin{remark}
The asymptotics
\[ \lambda_j(\Lh) = b_0 \hbar + \nu(0) \hbar^{3/2} + (E_j + c_0) \hbar^{2} + o(\hbar^2) \]
were unknown before, except in the special $3d$-case $M= \R^3$ in \cite{Birkhoff3D}. 
\end{remark}

\subsection{Related questions and perspectives}

In this paper, we are restricted to energies $\lambda < b_1 \hbar$, and as mentionned in Remark \ref{remark.b1.petit}, the threshold $b_1 > b_0$ is limited by three conditions, including the non-resonance one:
\[ \forall q \in \Omega \,, \quad \forall \alpha \in \Z^s \,, \quad 0<\vert \alpha \vert< \run \Rightarrow \sum_{j=1}^s \alpha_j \beta_j(q) \neq 0 \,. \]
It would be interesting to study the influence of resonances between the functions $\beta_j$ on the spectrum of $\Lh$. Maybe a Grushin reduction method could help, as in \cite{HelKor14-2} for instance. A Birkhoff normal form was given in \cite{CharlesVuNgoc} for a Schrödinger operator $-\hbar^2 \Delta + V$ with resonances, but the situation is somehow simpler, since the analogues of $\beta_j(q)$ are independent of $q$ in this context.

\medskip

We are also restricted by the existence of Darboux coordinates $\varphi$ on $(\Sigma,B)$, such that $\varphi^* B = \dd \eta \wedge \dd y$. Indeed, the coordinates $(y,\eta)$ on $\Sigma$ are necessary to use the Weyl quantization. To study the influence of the global geometry of $B$, one should consider another quantization method for the presymplectic manifold $(\Sigma,B)$. In the symplectic case, for instance in dimension $d=2$, a Toeplitz quantization may be useful. This quantization is linked to the complex structure induced by $B$ on $\Sigma$, and the operator $\Lh$ can be linked with this structure in the following way:
\[\Lh = 4 \hbar^2 \left( \dbar + \frac{i}{2\hbar} A \right)^* \left( \dbar + \frac{i}{2\hbar} A \right) + \hbar B = 4 \hbar^2 \dbar_A^* \dbar_A + \hbar B,  \]
with
\[A = A_1 + i A_2\,, \quad B= \partial_1 A_2 - \partial_2 A_1\,, \quad 2 \dbar = \partial_1 + i \partial_2 \,. \]
In \cite{Prieto}, this is used to compute the spectrum of $\Lh$ on a bidimensional Riemann surface $M$ with constant curvature and constant magnetic field. See also the recent papers \cite{Charles20,Kor20} where semi-excited states for constant magnetic fields in higher dimensions are considered.

\medskip

If the 2-form $B$ is not exact, we usually consider a Bochner Laplacian on the $p$-th tensor product of a complex line bundle $L$ over $M$, with curvature $B$. This Bochner Laplacian $\Delta_p$, depends on $p \in \N$, and the limit $p \rightarrow + \infty$ is interpreted as the semi-classical limit. $\Delta_p$ is a good generalization of the magnetic Laplacian because \textit{locally} it can be written $\frac{1}{\hbar^2}(i\hbar \nabla + \A)^2$, where the potential $\A$ is a local primitive of $B$, and $\hbar = p^{-1}$. For details, we refer to the recent articles \cite{Kor18}, \cite{Kor19}, \cite{MaSa18}, and the references therein. In \cite{Kor19}, Kordyukov constructed quasimodes for $\Delta_p$ in the case of a symplectic $B$ and discrete wells. He proved expansions:
$$\lambda_j(\Delta_p) \sim \sum_{\ell \geq 0} \alpha_{j \ell} p^{-\ell/2}.$$
Our work also gives such expansions for $\Delta_p$ as explained in \cite{Morin-Bochner}.

\medskip

In this paper, we only mentionned the study of the eigenvalues of $\Lh$: What about the eigenfunctions ? WKB expansions for the $j$-th eigenfunction were constructed on $\R^2$ in \cite{BonRay19}, and on a $2$-dimensional Riemannian manifold in \cite{Tho}. We do not know how to construct magnetic WKB solutions in higher dimensions. This article suggests that the directions corresponding to the kernel of $B$ could play a specific role.

\medskip
 
An other related question is the decreasing of the real eigenfunctions. Agmon estimates only give a $\grandO(e^{-c/\sqrt{\hbar}})$ decay outside any neighborhood of $q_0$, but 2D WKB suggest a $\grandO(e^{-c/\hbar})$ decay. In the recent paper \cite{BonRayVN19}, Bonthonneau, Raymond and V\~u Ng\d{o}c proved this on $\R^2$, using the FBI Transform to work on the phase space $T^* \R^2$. This kind of question is motivated by the study of the tunneling effect: The exponentially small interaction between two magnetic wells for example.

\medskip


In this paper, we only have investigated the spectral theory of the stationary Schrödinger equation with a pure magnetic field ; it would be interesting to describe the long-time dynamics of the full Schrödinger evolution, as was done in the Euclidean 2D case by Boil and V\~u Ng\d{o}c in \cite{Boil}.

\medskip

Finally, it would be interesting to study higher Landau levels and the effect of resonances in our normal forms, as was done by Charles and V\~u Ng\d{o}c in \cite{CharlesVuNgoc} for an electric Schrödinger operator $-\hbar^2 \Delta + V$.

\subsection{Structure of the paper}

In section \ref{sec.2} we prove Theorem \ref{thm.H.rond.Phi}, reducing the symbol $H$ of $\Lh$ on a neighborhood of $\Sigma = H^{-1}(0)$. In section \ref{sec.3} we construct the normal form, first in a space of formal series (section \ref{sec.3.formel}), and then the quantized version $\Nh$ (section \ref{sec.3.Quantized}). In section \ref{sec.4} we prove Theorem \ref{thm.main.first.bnf}. For this we describe the spectrum of $\Nh$ (section \ref{sec.4.Nh}), then we prove microlocalization properties on the eigenfunctions of $\Lh$ and $\Nh$ (section \ref{sec.4.microloc}), and finally we compare the spectra of $\Lh$ and $\Nh$ (section \ref{sec.4.conclusion}).

\medskip

In section \ref{Chap4-sec4} we focus on Theorem \ref{Chap4-Thm-spec-Nh0} which describes the spectrum of the effective operator $\Nh^{[1]}$. In \ref{Chap4-sec4-1} we reduce its symbol, in \ref{Chap4-sectionSecondFormalBNF} we construct a second formal Birkhoff normal form, and in \ref{Chap4-sec4-3} the quantized version $\Mh$. In \ref{Chap4-sec4-4} we compare the spectra of $\Nh^{[1]}$ and $\Mh$.

\medskip

Finally, sections \ref{sec.Coro1} and \ref{sec.Coro2} are dedicated to the proofs of Corollaries \ref{cor.reduction.Mh0} and \ref{cor.eigenvalue.asymptotics} respectively.

\section{Reduction of the principal symbol $H$} \label{sec.2}

\subsection{Notations}

$\Lh$ is a $\hbar$-pseudodifferential operator on $M$ with principal symbol $H$:
\[ H(q,p) = \vert p - A_q \vert^2_{g_q^*}\,, \quad p \in T_q^*M \,, \quad q \in M \,. \]
Here, $T^*M$ denotes the cotangent bundle of $M$, $p \in T_q^*M$ is a linear form on $T_qM$. The scalar product $g_q$ on $T_qM$ induces a scalar product $g_q^*$ on $T_q^*M$, and $\vert \cdot \vert_{g_q^*}$ denotes the associated norm. In this section we prove Theorem \ref{thm.H.rond.Phi}, reducing $H$ on a neighborhood of its minimum:
\[ \Sigma = \lbrace (q,p) \in T^*M\,, q \in \Omega \,, \quad p= A_q \rbrace \,. \]
Recall that $\Omega$ is a (small) neighborhood of $q_0 \in M \setminus \partial M$. We will construct canonical coordinates $(z,w,v) \in \R^{2d}$ with:
\[ z=(x,\xi) \in \R^{2s} \,, \quad w=(y,\eta) \in \R^{2s} \,, \quad v=(t,\tau) \in \R^{2k}\,. \]
$\R^{2d}$ is endowed with the canonical symplectic form
\[ \omega_0 = \dd \xi \wedge \dd x + \dd \eta \wedge \dd y + \dd \tau \wedge \dd t \,. \]
We will identify $\Sigma$ with
\[ \Sigma' = \lbrace (x,\xi,y,\eta,t,\tau) \in \R^{2d} \,, \quad x=\xi =0 \,, \quad \tau = 0 \rbrace = \R^{2s+k}_{(y,\eta,t)} \times \lbrace 0 \rbrace \,. \]

We will use several lemmas to prove Theorem \ref{thm.H.rond.Phi}. Before constructing the diffeomorphism $\Phi_1^{-1}$ on a neighborhood $U_1$ of $\Sigma$, we will restrict to $\Sigma$. Thus we need to understand the structure of $\Sigma$ induced by the symplectic structure on $T^*M$ (Section \ref{sec.2.Sigma}). Then we will construct $\Phi_1$ and finally prove Theorem \ref{thm.H.rond.Phi} (Section \ref{sec.2.Phi1}).

\subsection{Structure of $\Sigma$} \label{sec.2.Sigma}

Recall that on $T^*M$ we have the Liouville $1$-form $\alpha$ defined by
\[ \alpha_{(q,p)}( \mathcal{V}) = p( (\dd \pi)_{(q,p)} \mathcal{V}) \,, \quad \forall (q,p) \in T^*M \,, \quad \mathcal{V} \in T_{(q,p)}( T^*M) \,,\]
where $\pi : T^*M \rightarrow M$ is the canonical projection : $\pi(q,p) = q$, and $\dd \pi$ its differential. $T^*M$ is endowed with the symplectic form $\omega = \dd \alpha$. $\Sigma$ is a $d$-dimensional submanifold of $T^*M$ which can be identified with $\Omega$ using
\[ j : q \in \Omega \mapsto (q,A_q) \in \Sigma \,, \]
and its inverse, which is $\pi$. 

\begin{lemma}\label{lem.Sigma.B}
The restriction of $\omega$ to $\Sigma$ is $\omega_\Sigma = \pi^* B$.
\end{lemma} 

\begin{proof}
Fix $q \in \Omega$ and $Q \in T_qM$. Then
\[ (j^* \alpha)_q( Q ) = \alpha_{j(q)}( (\dd j) Q) = A_q((\dd \pi) \circ (\dd j) Q) = A_q(Q) \,, \]
because $\pi \circ j = \mathsf{id}$. Thus $j^*\alpha = A$ and $\alpha_\Sigma = \pi^* j^* \alpha = \pi^* A$. Taking the exterior derivative we get
\[ \omega_\Sigma = \dd \alpha_\Sigma = \pi^*( \dd A) = \pi^* B \,. \]
\end{proof}

Since $B$ is a closed $2$-form with constant rank equal to $2s$, $(\Sigma, \pi^*B)$ is a presymplectic manifold. It is equivalent to $(\Omega,B)$, using $j$. We recall the Darboux Lemma, telling that such a manifold is locally equivalent to $(\R^{2s+k}, \dd \eta\wedge \dd y)$.

\begin{lemma}\label{lem.Darboux}
Up to reducing $\Omega$, there exists an open subset $\Sigma'$ of $\R^{2s+k}_{(y,\eta,t)}$ and a diffeomorphism $\varphi : \Sigma' \rightarrow \Omega$ such that $\varphi^* B = \dd \eta \wedge \dd y$.
\end{lemma}

One can always take (any) coordinate system on $\Omega$. Up to working in these coordinates, it is enough to consider the case $M= \R^d$ with
\[ H(q,p) = \sum_{k,\ell =1}^d g^{k\ell}(q) (p_k - A_k(q))(p_\ell - A_{\ell}(q)) \,, \quad (q,p) \in T^* \R^d \simeq \R^{2d} \]
to prove Theorem \ref{thm.H.rond.Phi}. This is what we will do. In coordinates, $\omega$ is given by
\[ \omega = \dd p \wedge \dd q = \sum_{j=1}^d \dd p_j \wedge \dd q_j \]
and $\Sigma$ is the submanifold
\[ \Sigma = \lbrace (q, \A(q))\,, q \in \Omega \rbrace \subset \R^{2d} \,, \]
and $j \circ \varphi : \Sigma' \rightarrow \Sigma$.

\medskip

In order to extend $j \circ \varphi$ to a neighborhood of $\Sigma'$ in $\R^{2d}$ \textit{in a symplectic way}, it is convenient to split the tangent space $T_{j(q)} ( \R^{2d})$ according to tangent and normal directions to $\Sigma$. This is the purpose of the following two lemmas.

\begin{lemma}
Fix $j(q) = (q, \A(q)) \in \Sigma$. Then the tangent space to $\Sigma$ is 
\[ T_{j(q)} \Sigma = \lbrace (Q,P) \in \R^{2d}, \quad P = \nabla_q \A \cdot Q \rbrace\,. \]
Moreover, the $\omega$-orthogonal $T_{j(q)}\Sigma^\perp$ is
\[ T_{j(q)} \Sigma ^{\perp} = \lbrace (Q,P) \in \R^{2d}, \quad P = (\nabla_q \A)^T \cdot Q \rbrace \,.\]
Finally, 
\[ T_{j(q)} \Sigma \cap T_{j(q)} \Sigma^{\perp} = \mathsf{Ker}(\pi^* B) \,. \]
\end{lemma}

\begin{proof}
$\Sigma$ being the graph of $q \mapsto \A(q)$, its tangent space is the graph of the differential $Q \mapsto (\nabla_q \A) \cdot Q$. In order to caracterize $T \Sigma^\perp$, note that the symplectic form $\omega = \dd p \wedge \dd q$ is defined by
\[ \omega_{(q,p)}((Q_1,P_1),(Q_2,P_2)) = \langle P_2, Q_1 \rangle - \langle P_1, Q_2 \rangle \,, \]
where $\langle \cdot , \cdot \rangle$ denotes the usual scalar product on $\R^d$. Thus,
 \begin{align*}
(Q,P) \in T_{j(q)} \Sigma^{\perp} &\Longleftrightarrow \forall Q_0 \in \R^d, \quad \omega_{j(q)}((Q_0, \nabla_q \A \cdot Q_0), (Q,P))=0\\
&\Longleftrightarrow \forall Q_0 \in \R^d, \quad \langle P, Q_0 \rangle - \langle (\nabla_q \A) \cdot Q_0 , Q \rangle = 0\\
&\Longleftrightarrow \forall Q_0 \in \R^d, \quad \langle P -  (\nabla_q \A)^T \cdot Q , Q_0 \rangle =0\\
&\Longleftrightarrow P = (\nabla_q \A)^T \cdot Q.
\end{align*}
Finally, with Lemma \ref{lem.Sigma.B} we know that the restriction of $\omega$ to $T \Sigma$ is given by $\pi^*B$. Hence, $T_{j(q)} \Sigma \cap T_{j(q)} \Sigma^\perp$ is the set of $(Q,P) \in T_{j(q)} \Sigma$ such that
\[ \pi^* B((Q,P),(Q_0,P_0)) = 0 \,, \quad \forall (Q_0,P_0) \in T_{j(q)} \Sigma \,. \]
It is the kernel of $\pi^*B$.
\end{proof}

Now we define specific basis of $T_{j(q)} \Sigma$ and its orthogonal. Since $\B(q)$ is skew-symmetric with respect to $g$, there exist orthonormal vectors $\ub_1(q)\,, \vb_1(q)\,, \cdots \,, \ub_s(q) \,, \vb_s(q) \,, \wb_1(q) \,, \cdots \,, \wb_s(q) \in \R^d$ such that
\begin{equation}
\begin{cases}
\B \ub_j = - \beta_j \vb_j, \quad 1 \leq j \leq s, \\
\B \vb_j = \beta_j \ub_j, \quad 1 \leq j \leq s, \\
\B \wb_j = 0, \quad 1 \leq j \leq k.
\end{cases}
\end{equation}
Moreover, these vectors are smooth functions of $q$ because the non-zero eigenvalues $\pm i \beta_j(q)$ are simple. They define a basis of $\R^d$. Define the following $\omega$-orthogonal vectors to $\Sigma$:
\begin{equation}\label{Chap4-eq-def-fb}
\begin{cases}
\fb_j(q) := \frac{1}{\sqrt{\beta_j(q)}}( \ub_j(q), (\nabla_q \A)^T \cdot \ub_j(q)), \quad 1 \leq j \leq s,\\
\fb_j'(q) := \frac{1}{\sqrt{\beta_j(q)}}( \vb_j(q), (\nabla_q \A)^T \vb_j(q)), \quad 1 \leq j \leq s.
\end{cases}
\end{equation}
These vectors are linearly independent and
\[T_{j(q)} \Sigma^{\perp} =  K \oplus F \,, \]
with
\[ K = \mathsf{Ker}(\pi^* B), \quad F = \mathsf{Vect}( \fb_1, \fb_1', \cdots, \fb_s, \fb_s') \,. \]
Similarily, the tangent space $T_{j(q)} \Sigma$ admits a decomposition
\[ T_{j(q)} \Sigma = E  \oplus K \]
defined as follows. The map $\varphi : \Sigma' \rightarrow \Sigma$ from Lemma \ref{lem.Darboux} satisfies $\varphi^*( \pi^* B) = \dd \eta \wedge \dd y$. Thus its differential maps the kernel of $B$ on the kernel of $\dd \eta \wedge \dd y$:
\[ K = = \lbrace (\dd \varphi)_q ( 0, T) ; \quad T \in \R^k \rbrace \]
A complementary space of $K$ in $T \Sigma$ is given by
\begin{equation}
E := \lbrace (\dd \varphi)_q (W, 0) ; \quad W \in \R^{2s} \rbrace.
\end{equation}

\begin{lemma}
Fix $j(q) = (q,\A(q)) \in \Sigma$. Then we have the following decomposition:
{\setstretch{0.3}
$$T_{j(q)}(\R^{2d})=\begin{array}{c}
\ \hspace{0.2cm}T \Sigma^{\perp}\\
\ \hspace{0.2cm}\overbrace{}{}\\
E\oplus K\oplus F\oplus L\\
\underbrace{}{}\hspace{1.5cm}\ \\
T \Sigma \hspace{1.5cm}\ \\ 
\end{array}$$
}
where $L$ is a Lagrangian complement of $K$ in $(E \oplus F)^\perp$.
\end{lemma}

\begin{proof}
We have $T \Sigma + T \Sigma^{\perp} = E \oplus K \oplus F,$ and the restriction of $\omega = \dd p \wedge \dd q$ to this space has kernel $K = T \Sigma \cap T \Sigma^\perp$. Hence, the restriction $\omega_{E \oplus F}$ of $\omega$ to $E \oplus F$ is non-degenerate. So is its orthogonal $(E \oplus F)^\perp$, which is a symplectic vector space for $\omega$, with dimension $2s-4s=2k$, and we have:
$$T_{j(q)} \R^{2d} = (E \oplus F) \oplus (E \oplus F)^{\perp}.$$
$K$ is a Lagrangian subspace of $(E \oplus F)^\perp$. It admits a complementary Lagrangian : A subspace $L$ of $(E \oplus F)^\perp$ with dimension $k$ such that $\omega_L =0$, and $(E \oplus F)^\perp = K \oplus L$.
\end{proof}

\begin{remark}
A basis $(\gb_j)$ of $L$ is defined by:
\begin{equation}\label{eq.gbj}
\omega( \gb_j, \dd \varphi(0,T)) = T_j \,, \quad \forall T \in \R^k \,.
\end{equation}
Indeed, the decomposition $(E \oplus F)^{\perp} = K \oplus L$ yield a bijection between $L$ and the dual $K^*$. This bijection is $\gb \mapsto \omega( \gb, \cdot )$. The linear form
\[ \dd \varphi(0,T) \in K \mapsto T_j \in \R \]
is represented by $\gb_j$.
\end{remark}

\subsection{Construction of $\Phi_1$ and proof of Theorem \ref{thm.H.rond.Phi}}\label{sec.2.Phi1}

We identified the "curved" manifold $\Sigma$ with an open subset $\Sigma'$ of $\R^{2s+k}$ using $j \circ \varphi$. Moreover, we did this in such a way that $(j \circ \varphi)^*B = \dd \eta \wedge \dd y$. In this section we prove that we can identify a whole neighborhood of $\Sigma$ in $\R^{2d}_{(q,p)}$ with a neighborhood of $\Sigma'$ in $\R^{4s+2k}_{(z,w,v)}$, via a symplectomorphism $\Phi_{\un}$.

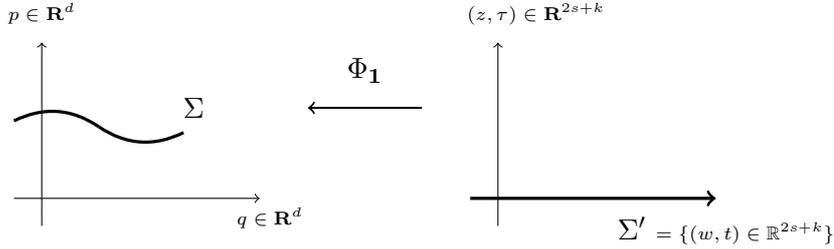
\begin{center}
\begin{tikzpicture}[scale=1]

\node (5) at (11/2,0.8) [] {};
\node (6) at (9,0.8) [] {};
\node (7) at (6,0.3) [] {};
\node (8) at (6,3) [] {};
\draw[->] (5) -- (6);
\draw[->] (7) -- (8);
\node at (12/2,13/4) [] {\begin{tiny}$p\in\mathbf{R}^d$\end{tiny}};
\node at (9,0.55) [] {\begin{tiny}$q\in\mathbf{R}^d$\end{tiny}};
\node (B1) at (11/2,7/4) [] {};
\node (B2) at (27/4,7/4) [shape=coordinate] {};
\node (B3) at (8,7/4) [] {};
\node at (8,2) [] {$\Sigma$};
\draw[very thick,bend left=30] (B1) to (B2);
\draw[very thick,bend right=30] (B2) to (B3);
\draw[<-,thick] (19/2,2) -- (22/2,2);
\node at (41/4,5/2) [] {$\Phi_\un$};
\node (9) at (23/2,0.8) [] {};
\node (10) at (15,0.8) [] {};
\node (11) at (12,0.3) [] {};
\node (12) at (12,3) [] {};
\draw[->,very thick] (9) -- (10);
\draw[->] (11) -- (12);
\node at (25/2,13/4) [] {\begin{tiny}$(z,\tau)\in\mathbf{R}^{2s+k}$\end{tiny}};
\node at (15,3/8) [] {$\Sigma'$ \begin{tiny}$= \lbrace (w,t)\in\mathbb{R}^{2s+k} \rbrace$\end{tiny}};
\end{tikzpicture}
\end{center}

\begin{lemma}
There exists a diffeomorphism $$\Phi_{\mathbf{1}} : U_{\mathbf{1}}' \subset \R^{2s+2k+2s}_{(w,t,\tau,z)} \rightarrow U_{\mathbf{1}} \subset \R^{2d}_{(q,p)},$$
such that $\Phi_{\mathbf{1}}^* \omega = \omega_0$ and $\Phi_{\mathbf{1}}(w,t,0,0) = j \circ \varphi(w,t)$ ; and such that its differential at $(w,t,\tau=0,z=0) \in \Sigma'$ is
$$\dd \Phi_{\mathbf{1}}(W,T,\mathcal{T},Z) = \dd_{(w,t)} j \circ \varphi (W,T) + \sum_{j=1}^k \mathcal{T}_j \hat{\gb}_j(w,t) + \sum_{j=1}^s X_j \hat{\fb}_j(w,t) + \Xi_j \hat{\fb}'_j(w,t).$$
\end{lemma}

\begin{remark}
In this lemma we used the notations $Z = (X,\Xi)$ and $\hat \gb_j = \gb_j \circ \varphi$, $\hat \fb_j = \fb_j \circ \varphi$, $\hat \fb_j' = \fb_j' \circ \varphi$.
\end{remark}

\begin{proof}
We will first construct $\Phi$ such that $\Phi^* \omega_{\vert \Sigma '} = \omega_{0 \ \vert \Sigma'}$ only on $\Sigma' = \Phi^{-1}(\Sigma)$. Then, we will use to Weinstein theorem to slightly change $\Phi$ into $\Phi_{\mathbf{1}}$ such that $\Phi_\un ^* \omega = \omega_0$ on a neighborhood of $\Sigma'$.

\medskip

We define $\Phi$ by:
\[ \Phi(w,t,\tau,z) = j \circ \varphi (w,t) + \sum_{j=1}^k \tau_j \hat{\gb}_j(w,t) + \sum_{j=1}^s x_j \hat{\fb}_j(w,t) + \xi_j \hat{\fb}_j'(w,t)\,. \]
Its differential at $(w,t,0,0)$ has the desired form. Let us fix a point $(w,t,0,0) \in \Sigma'$ and compute $\Phi^* \omega$ at this point. By definition,
\[ \Phi^* \omega_{(w,t,0,0)}(\bullet, \bullet) = \omega_{j(q)}( (\dd \Phi) \bullet, (\dd \Phi) \bullet), \]
where $q=\varphi(w,t)$. Computing this $2$-form in the canonical basis of $\R^{4s+2k}$, amounts to computing $\omega$ on the vectors $\gb_j$, $\fb_j$, $\fb_j'$ and $\dd (j \circ \varphi)(W,T)$. By definition of $\fb_j$ we have
\begin{align*}
\omega(\fb_i, \fb_j) &= \frac{1}{\sqrt{\beta_i \beta_j}} \left( \langle (\nabla_q \A)^{\perp} \cdot \ub_j , \ub_i \rangle - \langle (\nabla_q \A)^{\perp} \cdot \ub_i, \ub_j \rangle \right)\\
&= \frac{1}{\sqrt{\beta_i \beta_j}} \langle (\nabla_q \A)^{\perp} - (\nabla_q \A)) \cdot \ub_j , \ub_i \rangle \\
&= \frac{1}{\sqrt{\beta_i \beta_j}} B(\ub_j, \ub_i) \\
&= \frac{1}{\sqrt{\beta_i \beta_j}} g( \ub_j , \B \ub_i) \\
&=0,
\end{align*}
because $\B \ub_i = - \beta_i \vb_i$ is orthogonal to $\ub_j$. Similarily we get
\[\omega(\fb_i,\fb_j') = \delta_{ij}, \quad \omega(\fb_i',\fb_j') = 0\,. \]
Moreover, $\gb_i \in L \subset F^\perp$ so
\[ \omega( \gb_i, \fb_j) = \omega (\gb_i, \fb_j') = 0 \,. \]
Since $L$ is Lagrangian we also have $\omega( \gb_i , \gb_j) =0$. The vector $\dd (j \circ \varphi)(W,T)$ is tangent to $\Sigma$ and $\fb_j$, $\fb_j' \in T \Sigma'$ so
\[ \omega( \fb_j, \dd (j \circ \varphi)(W,T)) = \omega( \fb_j', \dd (j\circ \varphi)(W,T)) = 0 \,. \]
Since $\gb_i \in L \subset E^\perp$ and using \eqref{eq.gbj} we have
\[ \omega(\gb_j, \dd (j \circ \varphi)(W,T))= \omega( \gb_j, \dd (j \circ \varphi) (0,T)) = T_j \,. \]
Finally, $(j\circ \varphi)^* \omega = \varphi^* B = \dd \eta \wedge \dd y$ donc
$$\omega( \dd (j \circ \varphi)(W,T), \dd (j \circ \varphi)(W',T')) = \dd \eta \wedge \dd y ((W,T),(W',T')) \,.$$
All these computations show that $(\Phi^* \omega)_{(w,t,0,0)}$ coincide with $\omega_0 = \dd \xi \wedge \dd x + \dd \eta \wedge \dd y + \dd \tau \wedge \dd t$. Thus $\Phi^* \omega = \omega_0$ on $\Sigma$. With Weinstein theorem \ref{thm.Weinstein}, we can change $\Phi$ into $\Phi_\un(w,t,\tau,z) = \Phi(w,t,\tau,z) + \grandO((z,\tau)^2)$ such that $\Phi_\un^* \omega = \omega_0$ on a neighborhood $U_\un'$ of $\Sigma'$. In particular, the differential of $\Phi_\un$ at $(w,t,0,0)$ coincides with the differential of $\Phi$.
\end{proof}

Finally, the following lemma concludes the proof of Theorem \ref{thm.H.rond.Phi}.

\begin{lemma}
The Hamiltonian $\widehat{H} = H \circ \Phi_{\mathbf{1}}$ has the following Taylor expansion:
$$\widehat{H}(w,t,\tau,x,\xi) = \langle \partial^2_{\tau} \widehat{H}(w,t,0) \tau, \tau \rangle + \sum_{j=1}^s \widehat{\beta}_j(w,t) (\xi_j^2  + x_j^2) + \grandO( (\tau,x,\xi)^3 ).$$
\end{lemma}

\begin{proof}
Let us compute the differential (and Hessian) of
\[ H(q,p) = \sum_{k, \ell =1}^d g^{k \ell}(q) (p_k - A_k(q))(p_{\ell}- A_{\ell}(q)) \]
at a point $(q,\A(q)) \in \Sigma$. First,
\begin{equation}
\nabla_{(q,p)} H \cdot (Q,P) = \sum_{k, \ell =1}^d 2 g^{k\ell}(q) (p_k - A_k(q))(P_{\ell} - \nabla_q A_{\ell} \cdot Q) + (p_k - A_k(q))(p_{\ell} - A_{\ell}(q)) \nabla_q g \cdot Q,
\end{equation}
and at $p= \A(q)$ the Hessian is:
\begin{equation}\label{Chap4-eq-Hessienne-de-H}
\langle \nabla^2_{j(q)} H \cdot (Q,P), (Q',P') \rangle = 2 \sum_{k, \ell = 1}^d g^{k \ell}(q)(P_k - \nabla_q A_k \cdot Q) (P_{\ell}' - \nabla_q A_{\ell} \cdot Q').
\end{equation}
We can deduce a Taylor expansion of $\widehat{H}(w,t,\tau,z)$ with respect to $(\tau,z)$ (with fixed $(q,\A(q)) = j \circ \varphi(w,t)$). First,
\[ \widehat{H}(w,t,0,0) = H(q,\A(q)) = 0 \,. \]
Then we can compute the (partial) differential:
\[ \partial_{\tau, z} \widehat{H}(w,t,0,0) \cdot (W,T) = \nabla_{j(q)} H \cdot \partial_{\tau,z} \Phi_{\mathbf{1}}(w,t,0,0) \cdot (W,T) = \nabla_{j(q)} H \cdot \dd (j\circ \varphi)(W,T) = 0 \,, \]
because $\dd (j \circ \varphi)(W,T) \in T_{j(q)} \Sigma$. The Taylor expansion of $\widehat{H}$ is thus:
\[\widehat{H}(w,t,\tau,z) = \frac{1}{2}\langle \partial^2_{\tau,z} \widehat{H}(w,t,0) \cdot (\tau, z), (\tau,z) \rangle + \grandO((\tau,z)^3)\,, \]
where $\partial_{\tau,z}^2 \widehat{H}$ is the partial Hessian with respect to $(\tau,z)$. We have
\[\partial^2_{\tau,z} \widehat{H} = (\partial_{(\tau,z)} \Phi_{\mathbf{1}})^T \cdot \nabla_{j(q)}^2 H \cdot (\partial_{(\tau,z)} \Phi_{\mathbf{1}})\,, \]
and computing the Hessian matrix amounts to compute $\nabla^2_{j(q)} H$ on the vectors $\gb_j$, $\fb_j$, and $\fb_j'$. If $(Q,P) \in T_{j(q)} \Sigma^\perp$, then $P = (\nabla_q A)^\perp \cdot Q$ so that, with \eqref{Chap4-eq-Hessienne-de-H},
\begin{align*}
\frac{1}{2}\nabla^2_{j(q)} H ((Q,P),(Q',P')) &= \sum_{k,\ell,i, j =1}^d g^{k \ell}(q) ( \partial_k A_j Q_j - \partial_j A_k Q_j)( \partial_{\ell} A_i Q'_i - \partial_i A_{\ell} Q'_i)\\
&=  \sum_{k,\ell,i,j} g^{k\ell}(q) B_{kj} Q_j B_{\ell i} Q_i'.
\end{align*}
But $\sum_k g^{k \ell} B_{kj} = \B_{\ell j}$ so
\begin{align*}
\frac{1}{2}\nabla^2_{j(q)} H ((Q,P),(Q',P')) &
= \sum_{i,j, \ell} B_{\ell i} ( \B_{\ell j} Q_j ) Q_i'\\
&=  B( \B \cdot Q, Q').
\end{align*}
In the special case $(Q,P) = \fb_j$ we have
\begin{align*}
\frac{1}{2} \nabla^2_{j(q)} H ( \fb_i, \fb_j) &= \frac{1}{\sqrt{\beta_i \beta_j}} B( \B \ub_i , \ub_j)\\
&= \frac{1}{\sqrt{\beta_i \beta_j}} g( \B \ub_i, \B \ub_j )\\
&= \sqrt{\beta_i \beta_j} g( \vb_i, \vb_j )\\
& = \sqrt{\beta_i \beta_j} \delta_{ij}.
\end{align*}
and similarily
\[ \frac{1}{2} \nabla^2_{j(q)} H (\fb_i', \fb_j') = \sqrt{\beta_i \beta_j} \delta_{ij}, \quad \frac{1}{2} \nabla^2_{j(q)} H (\fb_i, \fb_j') = 0 \,. \]
Finally, it remains to prove:
\begin{equation}\label{Chap4-eq007}
\nabla^2_{j(q)} H (\gb_i, \fb_j) = \nabla^2_{j(q)} H (\gb_i, \fb_j') = 0,
\end{equation}
to conclude that the Hessian of $\widehat{H}$ is
\[ \partial^2_{\tau,z} \widehat{H} (w,t,0,0) =
\begin{pmatrix}
 \partial^2_{\tau} \widehat{H}(w,t,0,0) & & & & & \\
 & \beta_1 & & & & \\
 & & \beta_1 & & & \\
 & & & \ddots & & \\
 & & & & \beta_s & \\
 & & & & & \beta_s &
\end{pmatrix}. \]
Actually, \eqref{Chap4-eq007} follows from
\begin{equation}\label{Chap4-eq008}
L \subset F^{\perp} = \left( T \Sigma^{\perp} \right)^{\perp H},
\end{equation}
where $\perp H$ denotes the orthogonal with respect to the quadratic form $\nabla^2 H$ (which is different from the symplectic orthogonal $\perp$). Indeed, to prove \eqref{Chap4-eq008} note that:
\begin{align*}
(Q,P) \in (T \Sigma^{\perp})^{\perp H} &\Longleftrightarrow \forall Q' \in \R^d, \quad \nabla^2 H ((Q,P), (Q', (\nabla_q \A)^T \cdot Q') = 0 \\
&\Longleftrightarrow \forall Q' \in \R^d, \quad \sum_{k, \ell, j} g^{k\ell} (P_k - \nabla_q A_k \cdot Q) B_{\ell j} Q_j' = 0\\
&\Longleftrightarrow \forall Q' \in \R^d, \quad \sum_{k,j} (P_k - \nabla_q A_k \cdot Q) \B_{kj} Q_j' = 0\\
&\Longleftrightarrow \forall Q' \in \R^d, \quad \langle P - \nabla_q \A \cdot Q, \B  Q' \rangle = 0\\
& \Longleftrightarrow \forall Q' \in \R^d, \quad \langle P, \B Q' \rangle - \langle Q , (\nabla_q \A)^T \cdot \B Q' \rangle = 0\\
& \Longleftrightarrow \forall Q' \in \R^d, \quad \omega( (Q,P), (\B Q', (\nabla_q \A)^T \cdot \B Q' )) = 0,
\end{align*}
and we have
\begin{align*}
F &= \lbrace (V, (\nabla_q \A)^{T} V) ; \quad V \in \mathsf{vect}(\ub_1, \vb_1, \cdots \ub_s, \vb_s) \rbrace\\
&= \lbrace (\B Q, (\nabla_q \A)^{T} \B Q) ; \quad Q \in \R^{d} \rbrace,
\end{align*}
because the vectors $\ub_j, \vb_j$ span the range of $\B$. Hence we have
$$(Q,P) \in (T \Sigma^{\perp})^{\perp H} \Longleftrightarrow (Q,P) \in F^{\perp}.$$

\end{proof}

\section{Construction of the normal form $\Nh$}\label{sec.3}

\subsection{Formal series}

Denote by $U = U_\un' \cap \lbrace x=\xi=0 \,, \tau = 0 \rbrace \subset \R^{2s+k}_{(w,t)} \times \lbrace 0 \rbrace$. We construct the Birkhoff normal form in the space
\[ \mathcal{E}_\un = \Cinf(U) [[x,\xi,\tau,\hbar]] \,.\]
It is a space of formal series in $(x,\xi,\tau,\hbar)$ with $(w,t)$-dependent coefficients. We see these formal series as Taylor series of symbols, which we quantify using the Weyl quantization. Given a $\hbar$-pseudodifferential operator $\mathcal{A}_\hbar = \Op a_\hbar$ (with symbol $a_\hbar$ admitting an expansion in powers of $\hbar$ in some standard class), we denote by $[a_\hbar]$ or $\sigma^T(\mathcal{A}_\hbar)$ the Taylor series of $a_\hbar$ with respect to $(x,\xi,\tau)$ at $(x,\xi,\tau)=0$. Conversely, given a formal series $\rho \in \mathcal{E}_\un$, we can find a bounded symbol $a_\hbar$ such that $[ a_\hbar ]= \rho$. This symbol is not uniquely defined, but any two such symbols differ by $\grandO((x,\xi,\hbar)^\infty)$, uniformly with respect to $(w,t) \in U$.

\begin{remark}
We prove below that the eigenfunctions of $\Lh$ are microlocalized where $(w,t) \in U$ and $\vert (x,\xi) \vert \lesssim \hbar^{1/2}$, so that the remainders $\grandO((x,\xi,\hbar)^\infty)$ are negligible.
\end{remark}

\medskip

$\bullet$ In order to make operations on Taylor series compatible with the Weyl quantization, we endow $\mathcal{E}_\un$ with the Weyl product $\star$, defined by $\Op (a) \Op (b) = \Op  (a \star b)$. This product satisfies:
\[ a_1 \star a_2 = \sum_{k=0}^N \frac{1}{k!} \left( \frac{\hbar}{2i} \square \right)^k a_1(w,t,\tau,z) a_2(w',t',\tau',z')\vert_{w'=w, t'=t, \tau'=\tau, z'=z} + \grandO(\hbar^N) \]
where
\[ \square = \sum_{j=1}^s \left( \partial_{\eta_j} \partial_{y_j'}  - \partial_{y_j} \partial_{\eta_j'} \right) + \sum_{j=1}^s \left( \partial_{\xi_j} \partial_{x'_j} - \partial_{x_j} \partial_{\xi'_j} \right) + \sum_{j=1}^k \left( \partial_{\tau_j} \partial_{t'_j} - \partial_{t_j} \partial_{\tau_j'} \right). \]

\medskip

$\bullet$ The degree of a monomial is
\[ \text{deg}( x^\alpha \xi^{\alpha'} \tau^{\alpha''} \hbar^\ell ) = \vert \alpha \vert + \vert \alpha' \vert + \vert \alpha'' \vert + 2 \ell \,. \]
We denote by $\mathcal{D}_N$ the $\Cinf(U)$-module spanned by monomials of degree $N$, and $\grandO_N$ the $\Cinf(U)$-module of formal series with valuation $\geq N$. It satisfies
\[ \grandO_{N_1} \star \grandO_{N_2} \subset \grandO_{N_1+N_2} \,. \]
We denote commutators by
\[ [ \rho_1 , \rho_2] = \ad_{\rho_1} \rho_2 = \rho_1 \star \rho_2 - \rho_2 \star \rho_1 \,. \]
We have the formula
\[ [ \rho_1, \rho_2] = 2 \sinh \left( \frac{\hbar}{2i} \delta \right) \rho_1 \rho_2 \,. \]
In particular
\[ \forall \rho_1 \in \grandO_{N_1}, \quad \forall \rho_2 \in \grandO_{N_2}, \quad \frac{i}{\hbar}[\rho_1, \rho_2] \in \grandO_{N_1+N_2 -2} \,,\]
and $\frac i \hbar [ \rho_1, \rho_2 ] = \lbrace \rho_1, \rho_2 \rbrace + \grandO(\hbar^2)$. The Birkhoff normal form algorithm is based on the following lemma.

\begin{lemma}\label{lem.preliminary}
For $1 \leq j \leq s$, denote $z_j = x_j + i \xi_j$ and $\vert z_j \vert^2 = x_j^2 + \xi_j^2$.
\begin{enumerate}
\item Every series $\rho \in \mathcal{E}_\un$ satisfies
\[ \frac{i}{\hbar} \ad_{\vert z_j \vert^2} \rho = \lbrace \vert z_j \vert^2 , \rho \rbrace\,.\]
\item Let $0 \leq N < \run$. For every $R_N \in \mathcal{D}_N$, there exist $\rho_N, K_N \in \mathcal{D}_N$ such that
\[ R_N = K_N + \sum_{j=1}^s \widehat{\beta}_j(w,t) \frac i \hbar \mathsf{ad}_{\vert z_j \vert^2} \rho_N \,, \]
and $[ K_N, \vert z_j \vert^2 ] = 0$ for $1 \leq j \leq s$.
\item If $K \in \mathcal{E}_\un$, then $[K,\vert z_j \vert^2] =0$ for all $1 \leq j \leq s$ if and only if there exist a formal series $F \in \Cinf(U)[[I_1, \cdots, I_s, \tau, \hbar]]$ such that
\[ K = F(\vert z_1 \vert^2, \cdots, \vert z_s \vert^2, \tau, \hbar) \,. \]
\end{enumerate}
\end{lemma}

\begin{proof}
The first statement is a simple computation. For the second and the third, it suffices to consider monomials $R_N = c(w,t) z^\alpha \bar{z}^{\alpha'} \tau^{\alpha''} \hbar^\ell$. Note that:
\[ \ad_{\vert z_j \vert^2}(  c(w,t) z^{\alpha} \bar{z}^{\alpha'} \tau^{\alpha''} \hbar^{\ell} ) = (\alpha'_j - \alpha_j) c(w,t) z^{\alpha} \bar{z}^{\alpha'} \tau^{\alpha''} \hbar^{\ell}\,, \]
so that $R_N$ commutes with every $\vert z_j \vert^2$ ($1 \leq j \leq s$) if and only if $\alpha = \alpha'$, which amounts to say that $R_N$ is a function of $\vert z_j \vert^2$ and proves $(3)$. Moreover,
$$\sum_j \widehat{\beta}_j \ad_{\vert z_j \vert^2}(z^{\alpha} \bar{z}^{\alpha'} \tau^{\alpha''} \hbar^{\ell} ) = \langle \alpha' - \alpha, \widehat{\beta}\rangle z^{\alpha} \bar{z}^{\alpha'} \tau^{\alpha''} \hbar^{\ell},$$
where $\langle \gamma, \widehat{\beta} \rangle = \sum_{j=1}^s \gamma_j \widehat{\beta}_j(w,t)$. Under the assumption $\vert \alpha \vert + \vert \alpha' \vert + \vert \alpha '' \vert + 2 \ell < \run$, we have $\vert \alpha - \alpha' \vert < \run$ and by definition of $\run$ the function $\langle \alpha' - \alpha , \widehat{\beta}(w,t) \rangle$ cannot vanish for $(w,t) \in U$, unless $\alpha = \alpha'$. If $\alpha = \alpha'$, we chose $\rho_N = 0$ and $R_N = K_N$ commutes with $\vert z_j \vert^2$. If $\alpha \neq \alpha'$, we chose $K_N =0$ and 
$$\rho_N = \frac{c(w,t)}{\langle \alpha'-\alpha , \widehat{\beta} \rangle} z^{\alpha} \bar{z}^{\alpha'} \tau^{\alpha''} \hbar^{\ell},$$
and this proves $(2)$.
\end{proof}

\subsection{Formal Birkhoff normal form}\label{sec.3.formel}

In this section we construct the Birkhoff normal form at a formal level. We will work with the Taylor series of the symbol $H$ of $\Lh$, in the new coordinates $\Phi_\un$. According to Theorem \ref{thm.H.rond.Phi}, $\widehat{H} = H \circ \Phi_\un$ defines a formal series
\[ [ \widehat{H} ] = H_2 + \sum_{k \geq 3} H_k \,, \]
where $H_k \in \mathcal{D}_k$ and
\begin{equation}\label{Chap4-def-H2}
H_2 = \langle M(w,t) \tau, \tau \rangle + \sum_{j=1}^s \widehat{\beta}_j(w,t) \vert z_j \vert^2.
\end{equation}
At a formal level, the normal form can be stated as follows.

\begin{theorem}\label{thm.formal.bnf}
For every $\gamma \in \grandO_3$, there are $\kappa$, $\rho \in \grandO_3$ such that:
\[e^{\frac{i}{\hbar} \ad_{\rho}} (H_2 + \gamma) = H_2 + \kappa + \grandO_{\run},\]
where $\kappa$ is a function of harmonic oscillators: $\kappa = F(\vert z_1 \vert^2, \cdots, \vert z_s \vert^2, \tau, \hbar)$ with some $F \in \Cinf(U) [[I_1, \cdots, I_s, \tau, \hbar]]$. Moreover, if $\gamma$ has real-valued coefficients, then $\rho$, $\kappa$ and the remainder $\grandO_{\run}$ as well.
\end{theorem}

\begin{proof}
We prove this by induction on an integer $N \geq 3$. Assume that we found \newline $\rho_{N-1}\,, K_3 \,, \cdots \,, K_{N-1} \in \grandO_3$ with $[K_i, \vert z_j \vert^2] = 0$ for every $(i,j)$ and $K_i \in \mathcal{D}_i$ such that
\[ e^{\frac{i}{\hbar}\ad_{\rho_{N-1}}}(H_2 + \gamma) = H_2 + K_3 + \cdots + K_{N-1} + \grandO_N. \]
Rewritting the remainder as $R_N + \grandO_{N+1}$ with $R_N \in \mathcal{D}_N$ we have:
\[ e^{\frac{i}{\hbar}\ad_{\rho_{N-1}}}(H_2 + \gamma) = H_2 + K_3 + \cdots + K_{N-1} + R_N + \grandO_{N+1}\,. \]
We are looking for a $\rho' \in \grandO_N$. For such a $\rho'$ we apply $e^{\frac{i}{\hbar} \ad_{\rho'}}$:
\[e^{\frac{i}{\hbar} \ad_{\rho_{N-1}+ \rho'}} (H_2 + \gamma) = e^{\frac{i}{\hbar}\ad_{\rho'}}(H_2 + K_3 + \cdots + K_{N-1} + R_N + \grandO_{N+1}).\]
Since $\frac{i}{\hbar}\ad_{\rho'} : \grandO_k \rightarrow \grandO_{k+N-2}$ we have:
\begin{equation}\label{Chap4-eq012}
e^{\frac{i}{\hbar} \ad_{\rho_{N-1}+ \rho'}} (H_2 + \gamma) = H_2 + K_3 + \cdots + K_{N-1} + R_N + \frac{i}{\hbar}\ad_{\rho'}(H_2) + \grandO_{N+1}.
\end{equation}
The new term $\frac{i}{\hbar} \ad_{\rho'} (H_2) = - \frac{i}{\hbar} \ad_{H_2}(\rho')$ can still be simplified. Indeed by \eqref{Chap4-def-H2}, 
\begin{equation}\label{Chap4-eq011}
\frac{i}{\hbar} \ad_{H_2} (\rho') = \frac{i}{\hbar}\left[ \langle M(w,t) \tau, \tau \rangle , \rho' \right] + \sum_{j=1}^s \left( \widehat{\beta}_j \frac{i}{\hbar} \left[\vert z_j \vert^2, \rho' \right] + \vert z_j \vert^2 \frac{i}{\hbar}\left[ \widehat{\beta}_j, \rho' \right] \right),
\end{equation}
with :
\begin{equation*}
\frac{i}{\hbar}\left[ \widehat{\beta}_j, \rho' \right] = \sum_{i=1}^s \left( \frac{\partial \hat{\beta}_j}{\partial y_i} \frac{\partial \rho'}{\partial \eta_i} - \frac{\partial \hat{\beta}_j}{\partial \eta_i} \frac{\partial \rho'}{\partial y_i} \right) + \sum_{i=1}^k \frac{\partial \hat{\beta}_j}{\partial t_i} \frac{\partial \rho'}{\partial \tau_i} + \grandO_{N-1} = \grandO_{N-1},
\end{equation*}
because a derivation with respect to $(y,\eta,t)$ does not reduce the degree. Similarily,
\begin{align*}
\frac{i}{\hbar}\left[ \langle M(w,t) \tau, \tau \rangle , \rho' \right] &= \sum_{j=1}^k \left( \langle \partial_{t_j} M(w,t) \tau, \tau \rangle \frac{\partial \rho'}{\partial \tau_j} - \frac{\partial \langle M(w,t) \tau, \tau \rangle}{\partial \tau_j} \frac{\partial \rho'}{\partial t_j} \right) + \grandO_{N+1} \\ &= \grandO_{N+1}
\end{align*}
and thus \eqref{Chap4-eq011} becomes :
\begin{equation*}
\frac{i}{\hbar} \ad_{H_2}(\rho') = \sum_{j=1}^s \left( \hat{\beta}_j \frac{i}{\hbar}\ad_{\vert z_j \vert^2} (\rho')\right) + \grandO_{N+1}.
\end{equation*}
Using this formula in \eqref{Chap4-eq012} we get
\begin{equation*}
e^{\frac{i}{\hbar} \ad_{\rho_{N-1}+ \rho'}} (H_2 + \gamma) = H_2 + K_3 + \cdots + K_{N-1} + R_N - \sum_{j=1}^s  \hat{\beta}_j \frac{i}{\hbar}\ad_{\vert z_j \vert^2} (\rho') + \grandO_{N+1}.
\end{equation*}
Thus, we are looking for $K_N \,, \rho' \in \mathcal{D}_N$ such that
\[R_N = K_N + \sum_{j=1}^s  \hat{\beta}_j \frac{i}{\hbar}\ad_{\vert z_j \vert^2} (\rho'),\]
with $[K_N, \vert z_j \vert^2] = 0$. By Lemma \ref{lem.preliminary}, we can solve this equation provided $N < \run$, and this concludes the proof. Moreover, $\frac{i}{\hbar}\ad_{\vert z_j \vert^2}$ is a real endomorphism, so we can solve this equation on $\R$.
\end{proof}

\subsection{Quantizing the normal form}\label{sec.3.Quantized}

In this section we construct the normal form $\Nh$, quantizing Theorems \ref{thm.H.rond.Phi} and \ref{thm.formal.bnf}. We denote by $\Ih^{(j)}$ the harmonic oscillator with respect to $x_j$, defined by:
\[ \Ih^{(j)} = \Op ( \xi_j^2 + x_j^2) = -\hbar^2 \frac{\partial^2}{\partial x_j^2} + x_j^2\,. \]
$\Nh$ is a function of the harmonic oscillators $\Ih^{(j)}$ ($1 \leq j \leq s$), depending on parameters $(y,\eta,t,\tau)$. More precisely, we prove the following theorem.

\begin{theorem}\label{thm.first.quantized.normal.form}
There exist:
\begin{enumerate}
\item A microlocaly unitary operator $\mathsf{U}_\hbar : \Ld(\R^d_{x,y,t} ) \rightarrow \Ld(M)$, quantifying a symplectomorphism $\tilde \Phi_\un = \Phi_\un + \grandO((x,\xi,\tau)^2)$, microlocally on $U_\un ' \times U_\un$.

\item a function $f_\un^\star : \R^{2s+2k}_{y,\eta,t,\tau} \times \R^s_I \times [0,1]$ which is $\Cinf$ with compact support such that
\[ \fsUn(y,\eta,t,\tau,I,\hbar) \leq C \left((\vert I \vert + \hbar)^2 + \vert \tau \vert (\vert I \vert + \hbar) + \vert \tau \vert^3 \right)\,, \]

\item A $\hbar$-pseudodifferential operator $\Rh$, whose symbol is $\grandO((x,\xi,\tau,\hbar^{1/2})^{\run})$ on $U_\un'$, 
\end{enumerate}
such that
\[ \Uh^* \Lh \Uh = \Nh + \Rh \,, \]
with
\[\Nh = \Op \langle M(w,t) \tau, \tau \rangle + \sum_{j=1}^s \Ih^{(j)} \Op \widehat{\beta}_j(w,t) + \Op \fsUn(y,\eta,t,\tau, \Ih^{(j)}, \cdots , \Ih^{(s)}, \hbar) \,. \]
\end{theorem}

\begin{remark}
$\Uh$ is a Fourier integral operator quantizing the symplectomorphism $\tilde \Phi_\un$ (See \cite{Martinez, Zworski}). In particular, if $\mathcal{A}_\hbar$ is a pseudodifferential operator on $M$ with symbol $a_\hbar = a_0 + \grandO(\hbar^2)$, then $\Uh^* \mathcal{A}_\hbar \Uh$ is a pseudodifferential operator on $\R^d$ with symbol
\[ \sigma_\hbar = a_0 \circ \tilde{\Phi}_\un + \grandO(\hbar^2) \quad \text{on } U_\un' \,.\]
\end{remark}

\begin{remark}
Due to the parameters $(y,\eta,t,\tau)$ in the formal normal form, an additional quantization is needed, hence the $\Op f_\un^\star$ term. It is a quantization with respect to $(y,\eta,t,\tau)$ of an operator-valued symbol $\fsUn(y,\eta,t,\tau,\Ih^{(1)}, \cdots, \Ih^{(s)})$. Actually, this operator-symbol is simple since one can diagonalize it explicitly. Denoting by $h^j_{n_j}(x_j)$ the $n_j$-th eigenfunction of $\Ih^{(j)}$, associated to the eigenvalue $(2n_j-1)\hbar$, we have for $n \in \N^s$:
\[ \fsUn(y,\eta,t,\tau,\Ih^{(1)}, \cdots, \Ih^{(s)}, \hbar) h_n(x) = \fsUn(y,\eta,\tau,(2n-1)\hbar,\hbar) h_n(x) \,, \]
where $h_n(x) = h_{n_1}^1(x_1) \cdots h_{n_s}^s(x_s)$. Thus the operator $\Op \fsUn$ satisfies for $u \in \Ld (\R^{s+k}_{(y,t)})$, 
\[(\Op \fsUn ) u \otimes h_n = \left( \Op \fsUn (y,\eta,t,\tau, (2n-1)\hbar, \hbar) u \right) \otimes h_n \,. \]
\end{remark}

\begin{proof}
In order to prove Theorem \ref{thm.first.quantized.normal.form}, we first quantize Theorem \ref{thm.H.rond.Phi}. Using the Egorov Theorem, there exists a microlocally unitary operator $\Vh : \Ld(\R^d) \rightarrow \Ld(M)$ quantizing the symplectomorphism $\Phi_\un : U_\un ' \rightarrow U_\un$. Thus,
\[ \Vh^* \Lh \Vh = \Op (\sigma_\hbar)\,, \]
for some symbol $\sigma_\hbar$ such that
\[ \sigma_\hbar = \widehat{H} + \grandO(\hbar^2), \quad \text{on } U_\un' \,. \]
Then we use the following Lemma to quantize the formal normal form and conclude.
\end{proof}

\begin{lemma}
There exists a bounded pseudodifferential operator $\Qh$ with compactly supported symbol such that
\[ e^{\frac{i}{\hbar}\Qh} \Op (\sigma_{\hbar}) e^{-\frac{i}{\hbar}  \Qh} = \Nh + \Rh \,, \]
where $\Nh$ and $\Rh$ satisfy the properties stated in Theorem \ref{thm.first.quantized.normal.form}.
\end{lemma}

\begin{remark}
As explained below, the principal symbol $Q$ of $\Qh$ is $\grandO((x,\xi,\tau)^3)$. Thus, the symplectic flow $\varphi_t$ associated to the Hamiltonian $Q$ is $\varphi_t(x,\xi,\tau) = (x,\xi, \tau) + \grandO((x,\xi,\tau)^2)$. Moreover, the Egorov Theorem implies that $e^{-\frac{i}{\hbar}\Qh}$ quantizes the symplectomorphism $\varphi_1$. Hence, $\Vh e^{-\frac{i}{\hbar}\Qh}$ quantizes the composed symplectomorphism $\tilde{\Phi}_\un = \Phi_\un \circ \varphi_1 = \Phi_\un + \grandO((x,\xi,\tau)^2)$.
\end{remark}

\begin{proof}
The proof of this lemma follows the exact same lines as in the case $k=0$ (\cite{MagBNF} Theorem 4.1). Let us recall the main arguments. The symbol $\sigma_\hbar$ is equal to $\widehat{H} + \grandO(\hbar^2)$ on $U_\un'$. Thus, its associated formal series is $[ \sigma_\hbar ] = H_2 + \gamma$ for some $\gamma \in \grandO_3$. Using the Birkhoff normal form algorithm (Theorem \ref{thm.formal.bnf}), we get $\kappa$, $\rho \in \grandO_3$ such that
\[ e^{\frac{i}{\hbar}\ad_{\rho}} (H_2 + \gamma) = H_2 + \kappa + \grandO_{\run}\,. \]
If $Q_\hbar$ is a smooth compactly supported symbol with Taylor series $[Q_\hbar] = \rho$, then by the Egorov Theorem the operator
\begin{equation}\label{eq.eQ.sigma.eQ}
 e^{i\hbar^{-1} \Op Q_\hbar} \Op (\sigma_\hbar) e^{-i\hbar^{-1} \Op Q_\hbar} 
\end{equation}
has a symbol with Taylor series $H_2 + \kappa + \grandO_{\run}$. Since $\kappa$ commutes with the oscillator $\vert z_j \vert^2$, it can be written as
\[ \kappa = \sum_{2 \vert \alpha \vert + \vert \alpha' \vert + 2 \ell \geq 3} c_{\alpha \alpha' \ell}(w,t) \vert z_1 \vert^{2\alpha_1} \cdots \vert z_s \vert^{2\alpha_s} \tau_1^{\alpha_1'} \cdots \tau_k^{\alpha_k'} \hbar^{\ell}\,. \]
We can reorder this formal series using the monomials $(\vert z_j \vert^2)^{\star \alpha_j} = \vert z_j \vert^2 \star \cdots \star \vert z_j \vert^2$:
\[ \kappa = \sum_{2 \vert \alpha \vert + \vert \alpha' \vert + 2 \ell \geq 3} c^\star_{\alpha \alpha' \ell}(w,t) ( \vert z_1 \vert^2 )^{\star \alpha_1} \cdots (\vert z_s \vert^2)^{\star\alpha_s} \tau_1^{\alpha_1'} \cdots \tau_k^{\alpha_k'} \hbar^{\ell}\,. \]
If $f_\un^\star$ is a smooth compactly supported function with Taylor series
\[ [f_\un^\star] = \sum_{2 \vert \alpha \vert + \vert \alpha' \vert + 2 \ell \geq 3} c^\star_{\alpha \alpha' \ell}(w,t) I_1^{\alpha_1} \cdots I_s^{\alpha_s} \tau_1^{\alpha_1'} \cdots \tau_k^{\alpha_k'} \hbar^{\ell}\,, \]
then the operator \eqref{eq.eQ.sigma.eQ} is equal to
\[ \Op H_2 + \Op \fsUn(y,\eta,t,\tau,\Ih^{(1)}, \cdots, \Ih^{(s)}, \hbar) \]
modulo $\grandO_{\run}$.
\end{proof}

\section{Comparing the spectra of $\Lh$ and $\Nh$}\label{sec.4}

\subsection{Spectrum of $\Nh$}\label{sec.4.Nh}

In this section we describe the spectral properties of $\Nh$. Since it is a function of harmonic oscillators, we can diagonalize this operator in the following way. For $1 \leq j \leq s$ and $n_j \geq 1$, we recall that the $n_j$-th Hermite function $h^j_{n_j}(x_j)$ is an eigenfunction of $\Ih^{(j)}$:
\[ \Ih^{(j)} h^j_{n_j} = \hbar (2n_j-1) h^{j}_{n_j} \,. \]
Hence, the functions $(h_n)_{n \in \N^s}$ defined by
\[ h_n(x) = h_{n_1}^1 \otimes \cdots \otimes h^s_{n_s} (x) = h_{n_1}^1(x_1) \cdots h^s_{n_s}(x_s) \]
form a Hilbertian basis of $\Ld(\R^s_x)$. Thus, we can use this basis to decompose the space $\Ld(\R^{2s+k}_{x,y,t})$ on which $\Nh$ acts :
\[ \Ld(\R^{2s+k}) = \bigoplus_{n \in \N^s} \left( \Ld(\R^{s+k}_{y,t}) \otimes h_n \right) \,. \]
As a function of the harmonic oscillators, $\Nh$ preserves this decomposition, and
\[ \Nh = \bigoplus_{n \in \N^s} \Nh^{[n]}\,, \] 
where $\Nh^{[n]}$ is the pseudodifferential operator with symbol
\begin{equation}
N_\hbar^{[n]} = \langle M(w,t) \tau, \tau \rangle + \sum_{j=1}^s \widehat{\beta}_j(w,t) (2n_j + 1) \hbar + \fsUn(w,t,\tau,(2n-1)\hbar, \hbar) \,.
\end{equation}
In particular, the spectrum of $\Nh$ is given by
\[ \spectrum (\Nh) = \bigcup_{n \in \N^s}  \spectrum ( \Nh^{[n]} ) \,. \]
Moreover, as in the $k=0$ case, for any $b_1 >0$ there is a $N_{\text{max}}>0$ (independent of $\hbar$) such that
\[ \spectrum (\Nh) \cap (-\infty, b_1 \hbar) = \bigcup_{\vert n \vert \leq N_{\text{max}}}  \spectrum ( \Nh^{[n]} ) \cap (-\infty, b_1 \hbar) \,. \]
The reason is that the symbol $N_\hbar^{[n]}$ is greater than $b_1 \hbar$ for $n$ large enough. Finally, to prove our main theorem \ref{thm.main.first.bnf} it remains to compare the spectra of $\Lh$ and $\Nh$.

\subsection{Microlocalization of the eigenfunctions}\label{sec.4.microloc}

In this section we prove microlocalization results for the eigenfunctions of $\Lh$ and $\Nh$. These results are needed to show that the remainders $\grandO((x,\xi,\tau)^{\run})$ we got are small. More precisely, for each operator we need to prove that the eigenfunctions are microlocalized:
\begin{enumerate}
\item[•] inside $\Omega$ (space localization),
\item[•] where $\vert (x,\xi,\tau) \vert \lesssim \hbar^\delta$ for $\delta \in (0,\frac 1 2)$ (i.e. close to $\Sigma$).
\end{enumerate}
Fix $\tilde{b}_1$ such that 
\[K_{\tilde{b}_1} = \lbrace q \in M\,, b(q) \leq \tilde{b}_1 \rbrace \subset \subset \Omega \,. \]

\begin{lemma}[Space localization for $\Lh$]
Let $b_1 \in (b_0, \tilde{b}_1)$ and $\chi \in \Cinf_0(M)$ be a cutoff function such that $\chi=1$ on $K_{\tilde{b}_1}$. Then every normalized eigenfunction $\psi_\hbar$ of $\Lh$ associated with an eigenvalue $\lambda_\hbar \leq b_1 \hbar$ satisfies:
\[ \psi_\hbar = \chi_0 \psi_\hbar + \grandO(\hbar^{\infty}) \,, \]
where the $\grandO(\hbar^\infty)$ is independant of $(\lambda_\hbar, \psi_\hbar)$.
\end{lemma}

\begin{proof}
This follows from the Agmon estimates, 
\begin{equation}\label{eq.Agmon}
\Vert e^{d(q,K_{\tilde b_1})\hbar^{-1/4}} \psi_\hbar \Vert \leq C \Vert \psi_\hbar \Vert^2\,,
\end{equation}
as in the $k=0$ case (in \cite{MagBNF}). Indeed, from \eqref{eq.Agmon} we deduce
\[ \Vert (1-\chi_0) \psi \Vert \leq C e^{-\varepsilon \hbar^{-1/4}} \Vert \psi_\hbar \Vert \,, \]
as soon as $\chi_0 = 1$ on a $\varepsilon$-neighborhood of $K_{\tilde{b}_1}$.
\end{proof}

\begin{lemma}[Microlocalization near $\Sigma$ for $\Lh$]\label{Chap4-Lemme-MicrolocLh}
Let $\delta \in (0, \frac{1}{2})$, $b_1 \in (b_0, \tilde b_1)$ and 
\newline $\chi_1 \in \Cinf(T^*M)$ be a cutoff function equal to one on a neighborhood of $\Sigma$. Then every eigenfunction $\psi_\hbar$ of $\Lh$ associated with an eigenvalue $\lambda_\hbar \leq b_1 \hbar$ satisfies:
\[ \psi_\hbar = \Op \chi_1(\hbar^{-\delta}(q,p)) \psi_\hbar + \grandO(\hbar^\infty) \psi_\hbar \,,\]
where the $\grandO(\hbar^\infty)$ is in $\mathcal{L}(\Ld, \Ld)$ and independant of $(\lambda_\hbar,\psi_\hbar)$.
\end{lemma}

\begin{proof}
Let $g_\hbar \in \Cinf_0(\R)$ be such that
\[ g_{\hbar}(\lambda) =
\begin{cases}
1 \quad \text{si } \lambda \leq b_1 \hbar \,, \\
0 \quad \text{si } \lambda \geq \tilde{b}_1 \hbar \,.
\end{cases} \]
Then the eigenfunction $\psi_\hbar$ satisfies
\[ \psi_{\hbar} = g_{\hbar}(\lambda_{\hbar}) \psi_{\hbar} = g_{\hbar}(\Lh) \psi_{\hbar} \,. \]
Denoting $\chi = 1-\chi_1$, we will prove that
\begin{equation}\label{eq.micro.chi.g}
\Vert \Op \chi( \hbar^{-\delta} (q,p) ) g_{\hbar}(\Lh) \Vert_{\mathcal{L}(\Ld, \Ld)} = \grandO(\hbar^{\infty}),
\end{equation}
from which will follow $\psi_{\hbar} = \Op \chi_1( \hbar^{-\delta} (q,p)) \psi_{\hbar} + \grandO(\hbar^{\infty})\psi_\hbar$, uniformly with respect to $(\lambda_\hbar, \psi_\hbar)$.
\newline
\indent To lighten the notations, we define $\chi^w := \Op \chi(\hbar^{-\delta}(q,p))$. For every $\psi \in \Ld(M)$ we define $\varphi = g_\hbar(\Lh) \psi$. Then,
\begin{equation}\label{Chap4-eq-f40}
\langle \Lh \chi^w \varphi, \chi^w \varphi \rangle = \langle \chi^w \Lh \varphi, \chi^w \varphi \rangle + \langle \big{[} \Lh, \chi^w \big{]} \varphi, \chi^w \varphi \rangle.
\end{equation}
We will bound from above the right-hand side, and from below the left-hand side. First, since $g_\hbar(\lambda)$ is supported where $\lambda \leq \tilde{b}_1 \hbar$ we have,
\begin{equation}\label{Chap4-eq-f41}
\langle \chi^w \Lh \varphi, \chi^w \varphi \rangle \leq \tilde{b}_1 \hbar \Vert \chi^w \varphi \Vert^2.
\end{equation}
Moreover, the commutator $\big{[} \Lh, \chi^w \big{]}$ is a pseudodifferential operator of order $\hbar$, with symbol supported on $\mathsf{supp} \chi$. Hence, if $\underline{\chi}$ is a cutoff function having the same general properties of $\chi$, such that $\underline{\chi} = 1$ on $\mathsf{supp} \chi$, we have:
\begin{equation}\label{Chap4-eq-f42}
\langle \big{[} \Lh, \chi^w \big{]} \varphi, \chi^w \varphi \rangle \leq C \hbar \Vert \underline{\chi}^w \varphi \Vert \Vert \chi^w \varphi \Vert.
\end{equation}
Finally, the symbol of $\chi^w$ is equal to $0$ on a $\hbar^\delta$-neighborhood of $\Sigma$ ; and thus the symbol $\vert p - A(q) \vert^2$ of $\Lh$  is $\geq c \hbar^{2\delta}$ on the support of $\chi^w$. Hence the G\aa rding inequality yield
\begin{equation}
\langle \Lh \chi^w \varphi, \chi^w \varphi \rangle \geq c \hbar^{2 \delta} \Vert \chi^w \varphi \Vert^2.
\end{equation}
Using this last inequality in \eqref{Chap4-eq-f40}, and bounding the right-hand side with \eqref{Chap4-eq-f41} and \eqref{Chap4-eq-f42} we get to
\begin{equation*}
c \hbar^{2\delta} \Vert \chi^w \varphi \Vert^2 \leq \tilde{b}_1 \hbar \Vert \chi^w \varphi \Vert^2 + C \hbar \Vert \underline{\chi}^w \varphi \Vert \Vert \chi^w \varphi \Vert\,,
\end{equation*}
and we deduce that
$$ \Vert \chi^w \varphi \Vert \leq C \hbar^{1-2 \delta} \Vert \underline{\chi}^w \varphi \Vert.$$
Iterating with $\underline{\chi}$ instead of $\chi$, we finally get for arbitrarily large $N >0$,
$$\Vert \chi^w \varphi \Vert \leq C_N \hbar^N \Vert \varphi \Vert.$$
This is true for every $\psi$, with $\varphi = g_\hbar(\Lh) \psi$ and thus $\Vert \chi^w g_\hbar(\Lh) \Vert = \grandO(\hbar^\infty)$.
\end{proof}

\begin{lemma}[Microlocalization near $\Sigma$ for $\Nh$]\label{Chap4-Lemme-MicrolocNh1}
Let $\delta \in (0,\frac 1 2)$, $b_1 \in (b_0, \tilde{b}_1)$ and $\chi_1 \in \Cinf_0(\R^{2s+k}_{x,\xi,\tau})$ be a cutoff function equal to $1$ on a neighborhood of $0$. Then every eigenfunction $\psi_\hbar$ of $\Nh$ associated with an eigenvalue $\lambda_\hbar \leq b_1 \hbar$ satisfies:
\[ \psi_{\hbar} = \Op \chi_1( \hbar^{-\delta} (x,\xi,\tau)) + \grandO(\hbar^{\infty}) \psi_{\hbar}\,, \]
where the $\grandO(\hbar^\infty)$ is in $\mathcal{L}(\Ld, \Ld)$ and independant of $(\lambda_\hbar,\psi_\hbar)$.
\end{lemma}

\begin{proof}
Just as in the previous Lemma, it is enough to show that
\[ \Vert \chi^w g_\hbar(\Nh) \Vert = \grandO(\hbar^\infty) \]
where $\chi^w = \Op (1-\chi_1)(\hbar^{-\delta}(x,\xi,\tau))$. We prove this using the same method. If $\psi \in \Ld(\R^d)$ and $\varphi = g_\hbar(\Nh) \psi_\hbar$, 
\begin{equation}
\langle \Nh \chi^w \varphi , \chi^w \varphi \rangle = \langle \chi^w \Nh \varphi, \chi^w \varphi \rangle + \langle \big{[} \Nh, \chi^w \big{]} \varphi, \chi^w \varphi \rangle.
\end{equation}
One can bound the right-hand side from above as before, and found $\varepsilon >0$ such that
\begin{equation}
\langle \Nh \chi^w \varphi, \chi^w \varphi \rangle \geq (1-\varepsilon) \langle \mathcal{H}_2 \chi^w \varphi, \chi^w \varphi \rangle,
\end{equation}
with $\mathcal{H}_2 = \Op \left( \langle M(w,t)\tau,\tau \rangle + \sum \widehat \beta_j(w,t) \vert z_j \vert^2 \right)$. The symbol of $\chi^w$ vanishes on a $\hbar^\delta$-neighborhood of $x=\xi=\tau=0$. Thus we can bound from below the symbol of $\mathcal{H}_2$ and use the G\aa rding inequality:
\[ \langle \mathcal{H}_2 \chi^w \varphi, \chi^w \varphi \rangle \geq c \hbar^{2 \delta} \Vert \chi^w \varphi \Vert^2 \,. \]
We conclude the proof as in Lemma \ref{Chap4-Lemme-MicrolocLh}.
\end{proof}

\begin{lemma}[Space localization for $\Nh$] \label{Chap4-Lemme-MicrolocNh2} Let $b_1 \in (b_0, \tilde b _1)$ and $\chi_0 \in \Cinf_0(\R^{2s+k}_{y,\eta,t})$ be a cutoff function equal to $1$ on a neighborhood of $\lbrace \widehat{b}(y,\eta,t) \leq \tilde{b}_1 \rbrace$. Then every eigenfunction $\psi_\hbar$ of $\Nh$ associated with an eigenvalue $\lambda_\hbar \leq b_1 \hbar$ satisfies:
\[ \psi_{\hbar} = \Op \chi_0(w,t) \psi_{\hbar} + \grandO(\hbar^{\infty}) \psi_{\hbar} \,, \] 
where the $\grandO(\hbar^\infty)$ is in $\mathcal{L}(\Ld, \Ld)$ and independent of $(\lambda_\hbar,\psi_\hbar)$.
\end{lemma}

\begin{proof}
Every eigenfunction of $\Nh$ is given by $\psi_\hbar(x,y,t) = u_\hbar(y,t) h_n(x)$ for some Hermite function $h_n$ with $\vert n \vert \leq N_{\text{max}}$ and some eigenfunction $u_\hbar$ of $\Nh^{[n]}$. Thus, it is enough to prove the lemma for the eigenfunctions of $\Nh^{[n]}$. If $u_\hbar$ is such an eigenfunction, associated with an eigenvalue $\lambda_\hbar \leq b_1 \hbar$, then
\[ u_\hbar = g_\hbar(\Nh^{[n]}) u_\hbar \,. \]
We will prove that $\Vert \chi^w g_\hbar( \Nh^{[n]}) \Vert = \grandO(\hbar^{\infty})\,,$ with $\chi^w = \Op (1-\chi_0)$, which is enough to conclude. If $u \in \Ld(\R^{k+s}_{y,t})$ and $\varphi = g_\hbar(\Nh^{[n]}) u$, then
\begin{equation}\label{Chap4-eq-f50}
\langle \Nh^{[n]} \chi^w \varphi, \chi^w \varphi \rangle = \langle \chi^w \Nh^{[n]} \varphi, \chi^w \varphi \rangle + \langle \big{[} \Nh^{[n]} , \chi^w \big{]} \varphi, \chi^w \varphi \rangle.
\end{equation}
On the first hand we have the bound
\begin{equation}\label{Chap4-eq-f51}
\langle \chi^w \Nh^{[n]} \varphi, \chi^w \varphi \rangle \leq \tilde{b}_1 \hbar \Vert \chi^w \varphi \Vert^2 \,.
\end{equation}
On the other hand, the commutator $\big{[} \Nh^{[n]} , \chi^w \big{]}$ is a pseudodifferential operator of order $\hbar$ with symbol supported on $\mathsf{supp} \chi$. Moreover, its principal symbol is $\lbrace N_\hbar^{[n]}, \chi \rbrace$. From the definition of $N_\hbar^{[n]}$ we deduce
\begin{equation*}
\langle \big{[} \Nh^{[n]} , \chi^w \big{]} \varphi, \chi^w \varphi \rangle \leq \hbar (1 + C \hbar) \langle \underline{\chi}^w \vert \tau \vert^w \varphi, \chi^w \varphi \rangle,
\end{equation*}
where $\underline{\chi}$ has the same general properties as $\chi$, and is equal to $1$ on $\mathsf{supp} \chi$. By Lemma \ref{Chap4-Lemme-MicrolocNh1}, we can had a cutoff where $\vert \tau \vert \lesssim \hbar^{\delta}$ and we get
\begin{equation}\label{Chap4-eq-f52}
\langle \big{[} \Nh^{[n]} , \chi^w \big{]} \varphi, \chi^w \varphi \rangle \leq C \hbar^{1+\delta} \Vert \underline{\chi}^w \varphi \Vert \Vert \chi^w \varphi \Vert.
\end{equation}
Finally for $\varepsilon >0$ small enough we have the lower bound
\[\langle \Nh^{[n]} \chi^w \varphi , \chi^w \varphi \rangle \geq \hbar (\tilde{b}_1 + \varepsilon) \Vert \chi^w \varphi \Vert^2 \,, \]
because $N_\hbar^{[n]}(w,t) \geq \hbar \widehat{b}(w,t)$ and $\chi$ vanishes on a neighborhood of $\lbrace \widehat{b}(w,t) \leq \tilde{b}_1 \rbrace$. Using this lower bound in \eqref{Chap4-eq-f50}, and bounding the right-hand side with \eqref{Chap4-eq-f51} and \eqref{Chap4-eq-f52} we get
\begin{equation}
\hbar (\tilde{b}_1 + \varepsilon) \Vert \chi^w \varphi \Vert^2 \leq \hbar \tilde{b}_1 \Vert \chi^w \varphi \Vert^2 + C \hbar^{1+ \delta} \Vert \underline{\chi}^w \varphi \Vert \Vert \chi^w \varphi \Vert.
\end{equation}
Thus
\[ \varepsilon \Vert \chi^w \varphi \Vert \leq C \hbar^{ \delta} \Vert \underline{\chi}^w \varphi \Vert \,, \]
and we can iterate with $\underline{\chi}$ instead of $\chi$ to conclude.
\end{proof}

\subsection{Proof of Theorem \ref{thm.main.first.bnf}}\label{sec.4.conclusion}

To conclude the proof of Theorem \ref{thm.main.first.bnf}, it remains to show that
\[ \lambda_n(\Lh) = \lambda_n(\Nh) + \grandO(\hbar^{\run/2 - \kappa}) \]
uniformly with respect to $n \in [1, N_\hbar^{\text{max}}]$ with
\[ N_{\hbar}^{\mathsf{max}} = \max \lbrace n \in \N, \quad \lambda_n( \Lh) \leq b_1 \hbar \rbrace \,. \]
Here $\lambda_n(\mathcal{A})$ denotes the $n$-th eigenvalue of the self-adjoint operator $\mathcal{A}$, repeated with multiplicities.

\begin{lemma}\label{Chap4-lemme-ineq1}
One has
$$\lambda_n(\Lh) = \lambda_n(\Nh) + \grandO(\hbar^{\frac{\run}{2}- \varepsilon})$$
uniformly with respect to $n \in [1, N_{\hbar}^{\mathsf{max}}]$.
\end{lemma}

\begin{proof}
Let us focus on the "$\leq$" inequality. For $n \in [1, N_\hbar^{\text{max}}]$, denote by $\psi_n^\hbar$ the normalized eigenfunction of $\Nh$ associated with $\lambda_n(\Nh)$, and
\[ \varphi_n^\hbar = \Uh \psi_n^\hbar \,. \]
We will use $\varphi_n^\hbar$ as quasimode for $\Lh$. Let $N \in [1, N_\hbar^{\text{max}}]$ and
\[ V_N^{\hbar} = \vect \lbrace \varphi_n^{\hbar}; \quad 1 \leq n \leq N \rbrace \,. \]
For $\varphi \in V_N^\hbar$ we use the notation $\psi = \Uh^{-1} \varphi$. By Theorem \ref{thm.first.quantized.normal.form}, we have
\begin{equation}\label{Chap4-lemme-ineq1-eq01}
\langle \Lh \varphi , \varphi \rangle = \langle \Nh \psi, \psi \rangle + \langle \Rh \psi, \psi \rangle \leq \lambda_N(\Nh) \Vert \psi \Vert^2 + \langle \Rh \psi, \psi \rangle.
\end{equation}
According to lemmas \ref{Chap4-Lemme-MicrolocNh1} and \ref{Chap4-Lemme-MicrolocNh2}, $\psi$ is microlocalized where $\vert (x,\xi,\tau) \vert \leq \hbar^{\delta}$ and $(w,t) \in \lbrace \widehat{b}(w,t) \leq \tilde b _1 \rbrace \subset U$. But the symbol of $\Rh$ is such that $R_\hbar = \grandO((x,\xi,\tau,\hbar^{1/2})^{\run})$ for $(w,t) \in U$, so:
\begin{equation}\label{Chap4-lemme-ineq1-eq02}
\langle \Rh \psi, \psi \rangle = \grandO( \hbar^{\delta \run}) = \grandO(\hbar^{\frac{\run}{2}- \varepsilon}),
\end{equation}
for suitable $\delta \in (0,\frac{1}{2})$. By \eqref{Chap4-lemme-ineq1-eq01} and \eqref{Chap4-lemme-ineq1-eq02} we have
\begin{equation*}
\langle \Lh \varphi, \varphi \rangle \leq ( \lambda_N(\Lh) + C \hbar^{\frac{\run}{2}- \varepsilon}) \Vert \varphi \Vert^2, \quad \forall \varphi \in V_N^{\hbar}.
\end{equation*}
Since $V_N^\hbar$ is $N$-dimensional, the minimax principle implies that
\begin{equation}
\lambda_N(\Lh) \leq \lambda_N(\Nh) + C \hbar^{\frac{\run}{2}- \varepsilon}.
\end{equation}
The reversed inequality is proved in the same way: we take the eigenfunctions of $\Lh$ as quasimodes for $\Nh$, and we use the microlocalization lemma \ref{Chap4-Lemme-MicrolocLh}.
\end{proof}

\section{A second normal form in the case $k >0$}\label{Chap4-sec4}

In the last sections, we reduced the spectrum of $\Lh$ to the spectrum of a normal form $\Nh$. Moreover, if $b_1 > b_0$ is sufficiently close to $b_0$ then the spectrum of $\Nh$ in $(-\infty, b_1 \hbar)$ is given by the spectrum of $\Nh^{[1]}$, an $\hbar$-pseudodifferential operator on $\R^{s+k}_{(y,t)}$ with symbol
\begin{equation}
N_{\hbar}^{[1]} = \langle M(y,\eta,t) \tau, \tau \rangle + \hbar \widehat{b}(y,\eta,t) + f^{\star}_{\mathbf{1}}(y,\eta,t, \tau, \hbar).
\end{equation}

In this section, we will construct a Birkhoff normal form again, to reduce the spectrum of $\Nh^{[1]}$ to an effective operator $\Mh$ on $\R^s_y$. In that purpose, in section \ref{Chap4-sec4-1} we will find new canonical variables $(\hat{t}, \hat{\tau})$ in which $N_{\hbar}^{[1]}$ is the perturbation of an harmonic oscillator. In sections \ref{Chap4-sectionSecondFormalBNF} and \ref{Chap4-sec4-3} we will construct the semiclassical Birkhoff normal form $\Mh$. In section \ref{Chap4-sec4-4} we will prove that the spectrum of $\Nh^{[1]}$ is given by the spectrum of $\Mh$.\\

In the following, we assume that $t \mapsto \widehat{b}(w,t)$ admits a non-degenerate minimum at $s(w)$ for $w$ in a neighborhood of $0$, and we denote by $(\nu_1^2(w), \cdots, \nu_k^2(w))$ the eigenvalues of the positive symmetric matrix:
$$M(w,s(w))^{1/2} \cdot \frac{1}{2} \partial_t^2 \hat{b}(w,s(w)) \cdot M(w,s(w))^{1/2}.$$
$\nu_1, \cdots, \nu_k$ are smooth non-vanishing functions in a neighborhood of $w=0$.

\subsection{Symplectic reduction of $N_{\hbar}^{[1]}$}\label{Chap4-sec4-1}

In this section we prove the following Lemma.

\begin{lemma}\label{Chap4-Lemma-Red-Nh0}
There exists a canonical (symplectic) transformation $\Phi_{\mathbf{2}} : U_{\mathbf{2}} \rightarrow V_{\mathbf{2}}$  between neighborhoods $U_{\mathbf{2}}$, $V_{\mathbf{2}}$ of $0 \in \R^{2s+2k}_{(y,\eta,t,\tau)}$ such that 
\begin{align*}
\hat{N}_{\hbar} := N_{\hbar}^{[1]} \circ \Phi_{\mathbf{2}} &= \hbar \widehat{b}(w,s(w)) + \sum_{j=1}^k \nu_j(w) \left( \tau_j^2 + \hbar t_j^2 \right) + \grandO(\vert t \vert^3 \vert \tau \vert^2) \\ & \quad + \grandO(\vert t \vert^3 \hbar) + \grandO(\hbar^2) + \grandO(\hbar \vert \tau \vert ) + \grandO(\vert \tau \vert^3) + \grandO( \vert t \vert \vert \tau \vert^2).
\end{align*}
\end{lemma}

\begin{proof}
We want to expan $\Nh^{[1]}$ near its minimum with respect to the variables $v=(t,\tau)$. First, from the Taylor expansion of $f^{\star}_{\mathbf{1}}$ we deduce:
\begin{align*}
N_{\hbar}^{[1]} &= \langle M(w,t) \tau, \tau \rangle + \hbar \widehat{b}(w,t) + \grandO(\hbar^2) + \grandO( \tau \hbar) + \grandO(\tau^3).
\end{align*}
We will Taylor-expan $t \mapsto \widehat{b}(w,t)$ on a neighborhood of its minimum point $s(w)$. In that purpose, we define new variables $(\tilde{y}, \tilde{\eta},\tilde{t},\tilde{\tau}) = \tilde{\varphi}(y,\eta,t,\tau)$ by:
$$\begin{cases}
\tilde{y} &= y - \sum_{j=1}^k \tau_j \nabla_{\eta} s_j(y,\eta) ,\\
\tilde{\eta} &= \eta +  \sum_{j=1}^k \tau_j \nabla_{y} s_j(y,\eta),\\
\tilde{t} &= t - s(y,\eta),\\
\tilde{\tau} &= \tau.
\end{cases}$$
Then $\tilde{\varphi}^* \omega_0 = \omega_0 + \grandO( \tau ).$
Using the Darboux-Weinstein Theorem \ref{thm.Weinstein}, we can make $\tilde{\varphi}$ symplectic on a neighborhood of $0$, up to a change of order $\grandO( \tau ^2)$. In these new variables, the symbol $\tilde{N}_{\hbar} := N_{\hbar}^{[1]} \circ \tilde{\varphi}^{-1}$ is:
\begin{align*}
\tilde{N}_{\hbar} &= \langle M \left[ \tilde{w} + \grandO( \tilde{\tau} ),\tilde{t}+ s(\tilde{w} + \grandO(\tilde{\tau})) \right] \tilde{\tau}, \tilde{\tau} \rangle + \hbar \widehat{b} \left[ \tilde{y} + \grandO(\tilde{\tau}), \tilde{\eta} + \grandO (\tilde{\tau}), s(\tilde{y}, \tilde{\eta})+ \tilde{t} + \grandO(\tilde{\tau})\right] \\ & \quad + \grandO(\hbar^2) + \grandO(\hbar \tilde{\tau})+ \grandO(\tilde{\tau}^3)\\
&= \langle M(\tilde{w},\tilde{t}+s(\tilde{w})) \tilde{\tau}, \tilde{\tau} \rangle + \hbar \widehat{b} \left[ \tilde{y},\tilde{\eta},s(\tilde{y},\tilde{\eta}) + \tilde{t}\right] + \grandO(\hbar^2) + \grandO(\hbar \tilde{\tau}) + \grandO(  \tilde{\tau} ^3).
\end{align*}
Then we remove the tildes and we expand this symbol in powers of $t$, $\tau$, $\hbar$. We get
\begin{align*}
\tilde{N}_{\hbar} &= \langle M(w,s(w)) \tau,\tau \rangle + \hbar \widehat{b} (w,s(w)) + \frac{\hbar}{2} \langle \partial_t^2 \widehat{b}(w,s(w)) t, t \rangle \\ & \quad + \grandO( \vert t \vert ^3 \hbar) + \grandO(\hbar^2) + \grandO(\hbar \vert \tau \vert) + \grandO( \vert \tau \vert^3) + \grandO(\vert t \vert \vert \tau \vert^2).
\end{align*}
Now, we want to coreduce the positive quadratic forms $M (w,s(w))$ and $\frac{1}{2}\partial^2_t \widehat{b} \left[w,s(w)\right]$. The reduction of quadratic forms in orthonormal coordinates implies that there exists a matrix $P(w)$ such that:
$$^t P \ M^{-1} P =  I, \quad \text{and } ^t P \ \frac{1}{2} \partial^2_t \widehat{b}\  P = \mathrm{diag}(\nu_1^2, ..., \nu_k^2).$$
We define the new coordinates $(\check{y}, \check{\eta}, \check{t}, \check{\tau}) = \check{\varphi}(y,\eta,t,\tau)$ by:
$$\begin{cases}
\check{t} &=   P(w)^{-1} t\\
\check{\tau} &= \ ^t P(w) \tau\\
\check{y} &= y + \ ^t [ \nabla_{\eta} (P^{-1} t) ] . ^t P  \tau \\
\check{\eta} &= \eta - \ ^t [ \nabla_y (P^{-1} t) ] . ^t P \tau,
\end{cases}$$
so that $\check{\varphi}^* \omega_0 - \omega_0 = \grandO ( \vert t \vert^2 + \vert \tau \vert )$. Again, we can make it symplectic up to a change of order $\grandO(\vert t \vert^3 + \vert \tau \vert^2)$ by Weinstein Lemma \ref{thm.Weinstein}. In these new variables, the symbol becomes (after removing the "checks"):
\begin{align*}
\check{N}_{\hbar} &= \hbar \widehat{b}(w,s(w)) + \sum_{j=1}^k \left( \tau_j^2 + \hbar \nu_j(w)^2 t_j^2 \right)  + \grandO(\vert t \vert^3 \vert \tau \vert^2) \\ & \quad + \grandO(\vert t \vert^3 \hbar) + \grandO(\hbar^2) + \grandO(\hbar \vert \tau \vert ) + \grandO(\vert \tau \vert^3) + \grandO( \vert t \vert \vert \tau \vert^2).
\end{align*}

The last change of coordinates $(\hat{y},\hat{\eta},\hat{t},\hat{\tau}) = \hat{\varphi}(y,\eta,t,\tau)$ defined by:
$$\begin{cases}
\hat{t}_j &=   \nu_j(w)^{1/2} t_j\\
\hat{\tau}_j &= \nu_j(w)^{-1/2} \tau_j\\
\hat{y}_j &= y_j +  \sum_{i=1}^k \nu_i^{-1/2} \tau_i \partial_{\eta_j} \nu_i ^{1/2} t_i \\
\hat{\eta} &= \eta - \sum_{i=1}^k  \nu_i^{-1/2} \tau_i \partial_{y_j} \nu_i^{1/2} t_i,
\end{cases}$$
is such that $\hat{\varphi}^* \omega_0 = \omega_0 + \grandO( \tau ),$ so it can be corrected modulo $\grandO( \vert \tau \vert^2)$ to be symplectic, and we get the new symbol:
\begin{align*}
\hat{N}_{\hbar} &=  \hbar \widehat{b} (w,s(w))  + \sum_{j=1}^k \nu_j(w) \left( \tau_j^2 + \hbar t_j^2 \right) + \grandO(\vert t \vert^3 \vert \tau \vert^2) \\ & \quad + \grandO(\vert t \vert^3 \hbar) + \grandO(\hbar^2) + \grandO(\hbar \vert \tau \vert ) + \grandO(\vert \tau \vert^3) + \grandO( \vert t \vert \vert \tau \vert^2).
\end{align*}
and Lemma \ref{Chap4-Lemma-Red-Nh0} is proved.
\end{proof}

\subsection{Formal second normal form} \label{Chap4-sectionSecondFormalBNF}

The harmonic oscillators occuring in $\hat{N}_{\hbar}$ are $$\Jh^{(j)} = \Oph(\hbar^{-1}\tau_j^2 + t_j^2), \quad 1 \leq j \leq k.$$ If we denote $$h= \sqrt{\hbar},$$ the symbol of $\Jh^{(j)}$ for the $h$-quantization is $\tilde{\tau}^2_j + t_j^2$. This is why we use the following mixed quantization:
$$\Opt(\mathsf{a})u(y_0,t_0) = \frac{1}{(2\pi \hbar)^{n-k}(2\pi \sqrt{\hbar})^k} \int e^{\frac{i}{\hbar}\langle y_0-y, \eta \rangle}e^{\frac{i}{\sqrt{\hbar}}\langle t_0-t , \tilde{\tau} \rangle} \mathsf{a}(\sqrt{\hbar},y,\eta,t,\tilde{\tau}) \dd y \dd \eta \dd t \dd \tilde{\tau}.$$
It is related to the $\hbar$-quantization by the relation $$\tau = h\tilde{\tau}, \quad h = \sqrt{\hbar}.$$ In other words, if $a$ is a symbol in some standard class $S(m)$, and if we denote:
$$\mathsf{a}(h,y,\eta,t,\tilde{\tau}) = a(h^2,y,\eta,t,h \tilde{\tau}),$$
then we have:
$$\Opt(\mathsf{a}) = \Oph(a).$$
However, if we take $\mathsf{a} \in S(m)$, then $\Opt(\mathsf{a})$ is not necessarily a $\hbar$-pseudodifferential operator, since the associated $a$ may not be bounded with respect to $\hbar$, and though it does not belong to any standard class. For instance, we have:
$$\partial_{\tau} a = \frac{1}{\sqrt{\hbar}} \partial_{\tilde{\tau}} \mathsf{a}.$$
But still $\Opt(\mathsf{a})$ is a $h$-pseudodifferential operator, with symbol:
$$\mathfrak{a}(h,y,\tilde{\eta},t,\tilde{\tau})=\mathsf{a}(h,y,h\tilde{\eta},t,\tilde{\tau}).$$
With this notation:
$$\Opt ( \mathsf{a} ) = \Oph (\mathfrak{a}).$$
Thus, in this sense, we can use the properties of $\hbar$-pseudodifferential and $h$-pseudodifferential operators to deal with our mixed quantization.\\

In our case, we have:
$$\Opt(\mathsf{N}_h) = \Oph(\widehat{N}_{\hbar}),$$
with
$$\mathsf{N}_h = h^2 \widehat{b}(w,s(w)) + h^2 \sum_{j=1}^k \nu_j(w) ( \tilde{\tau}_j^2 + t_j^2 )+ \grandO(h^2 \vert t \vert ^3) + \grandO(h^4) + \grandO(h^3 \vert \tilde{\tau} \vert) + \grandO (h^2 \vert t \vert \vert \tilde{\tau} \vert^2 ).$$

Let us construct a semiclassical Birkhoff normal form with respect to this quantization. We will work in the space of formal series 
$$\mathcal{E}_{\mathbf{2}}:= \mathcal{C}^{\infty}(\R^d_w)[[t,\tilde{\tau},h]]$$
endowed with the star product $\star$ adapted to our mixed quantization.  In other words
$$\Opt (\mathsf{a} \star \mathsf{b}) = \Opt(\mathsf{a}) \Opt (\mathsf{b}).$$
The change of variable $\tau = h\tilde{\tau}$ between the classical quantization and our mixed quantization yields the following formula for the star product:
\begin{align}\label{StarProductMixedQuantization}
 \mathsf{a} \star \mathsf{b} = \sum_{k \geq 0} \frac{1}{k!} \left( \frac{h}{2i}\right)^k  A_h(\partial) ^k (\mathsf{a}(h,y_1,\eta_1,t_1,\tilde{\tau}_1)\mathsf{b}(h,y_2,\eta_2,t_2,\tilde{\tau}_2))_{\vert (t_1,\tau_1,y_1,\eta_1) = (t_2,\tau_2,y_2,\eta_2)},
\end{align}
with
$$A_h(\partial) = \sum_{j=1}^k \frac{\partial}{\partial t_{1j}}\frac{\partial}{\partial \tilde{\tau}_{2j}} - \frac{\partial}{\partial t_{2j}} \frac{\partial}{\partial \tilde{\tau}_{1j}} +  h \sum_{j=1}^s \frac{\partial}{\partial y_{1j}} \frac{\partial}{\partial \eta_{2j}} - \frac{\partial}{\partial y_{2j}}\frac{\partial}{\partial \eta_{1j}}.$$
The degree function on $\mathcal{E}_{\mathbf{2}}$ is defined by:
$$\deg(t^{\alpha_1}\tilde{\tau}^{\alpha_2} h^{\ell}) = \vert \alpha_1 \vert + \vert \alpha_2 \vert + 2 \ell.$$
The degree of a general series depends on $w$. We denote by $\mathcal{D}_N$ the subspace spanned by monomials of degree $N$, and $\mathcal{O}_N$ the subspace of formal series with valuation at least $N$ on a neighborhood of $w=0$. For $\tau_1, \tau_2 \in \mathcal{E}_{\mathbf{2}}$, we define 
$$\mathsf{ad}_{\tau_1}(\tau_2) = [\tau_1,\tau_2] = \tau_1 \star \tau_2 - \tau_2 \star \tau_1.$$
Then, if $\tau_1 \in \grandO_{N_1}$ and $\tau_2 \in \grandO_{N_2}$,
$$\frac{i}{h}\ad_{\tau_1}(\tau_2) \in \grandO_{N_1+N_2-2}.$$
We denote
$$N_0 = \widehat{b}(w,s(w)) \in \mathcal{D}_{0} \quad \text{and} \quad N_2 = \sum_{j=1}^k \nu_j(w) \vert \tilde{v}_j \vert^2 \in \mathcal{D}_2,$$
with the notation $\tilde{v}_j = t_j+ i \tilde{\tau}_j$, so that
$$\frac{1}{h^2} \mathsf{N}_h = N_0 + N_2 + \grandO_3.$$
Then we can construct the following normal form. Recall that $r_{\mathbf{2}}$ is an integer chosen such that
$$\forall \alpha \in \Z^k, \quad 0 < \vert \alpha \vert < r_{\mathbf{2}}, \quad \sum_{j=1}^s \alpha_j \nu_j(0) \neq 0.$$
Moreover, this non-resonance relation at $w=0$ can be extended to a small neighborhood of $0$.

\begin{lemma}\label{Chap4-formalBNF2}
For any $\gamma \in \grandO_3$, there exist $\kappa, \tau$ $\in \grandO_3$ and $\rho \in \grandO_{r_{\mathbf{2}}}$ such that
\begin{align}\label{eqformelleBNF2}
e^{\frac{i}{h}\ad_{\tau}} \left( N_0 + N_2 + \gamma \right) = N_0 + N_2 + \kappa + \rho,
\end{align}
and $[\kappa,\vert \tilde{v}_j \vert^2]=0$ for $1 \leq j \leq k$.
\end{lemma}

\begin{proof}
We prove this result by induction. Assume that we have, for some $N>0$, a $\tau \in \grandO_3$ such that
$$e^{\frac{i}{h}\ad_{\tau}}(N_0+N_2+\gamma) = N_0 + N_2 + K_3 + ... + K_{N-1} + R_N + \grandO_{N+1},$$
with $R_N \in \mathcal{D}_N$ and $K_i \in \mathcal{D}_i$ such that $[K_i,\vert \tilde{v}_j \vert^2]=0$. We are looking for a $\tau_N \in \mathcal{D}_N$. For such a $\tau_N$, $\frac{i}{h}\ad_{\tau_N} : \grandO_j \rightarrow \grandO_{N+j-2}$ so:
$$e^{\frac{i}{h}\ad_{\tau+\tau_N}}(N_0+N_2+\gamma) = N_0 + N_2 + K_3 + ... + K_{N-1} + R_N + \frac{i}{h}\ad_{\tau_N} (N_0+N_2) + \grandO_{N+1}.$$
Moreover $N_0$ does not depend on $(t,\tau)$ so the expansion (\ref{StarProductMixedQuantization}) yields
$$\frac{i}{h}\ad_{\tau_N}(N_0) = h \sum_{j=1}^s \left( \frac{\partial \tau_N}{\partial y_j}\frac{\partial N_0}{\partial \eta_j} -\frac{\partial \tau_N}{\partial \eta_j}\frac{\partial N_0}{\partial y_j} \right) + \grandO_{N+6} = \grandO_{N+2},$$
and thus:
$$e^{\frac{i}{h}\ad_{\tau+\tau_N}}(N_0+N_2+\gamma) = N_0 + N_2 + K_3 + ... + K_{N-1} + R_N + \frac{i}{h}\ad_{\tau_N} (N_2) + \grandO_{N+1}.$$
So we are looking for $\tau_N,K_N \in \mathcal{D}_N$ solving the equation
\begin{align}\label{eq1442}
R_N = K_N + \frac{i}{h}\ad_{N_2} \tau_N + \grandO_{N+1}.
\end{align}
To solve this equation, we study the operator $\frac{i}{h}\ad_{N_2} : \grandO_N \rightarrow \grandO_N.$
$$ \frac{i}{h}\ad_{N_2}(\tau_N) = \sum_{j=1}^k \left( \nu_j(w) \frac{i}{h}\ad_{\vert \tilde{v}_j \vert^2} (\tau_N) + \frac{i}{h}  \ad_{\nu_j}(\tau_N) \vert \tilde{v}_j \vert^2 \right) ,$$
and since $\nu$ only depends on $w$, expansion (\ref{StarProductMixedQuantization}) yields:
$$\frac{i}{h} \ad_{\nu_i} (\tau_N) = \sum_{j=1}^s h \left( \frac{\partial \nu_i}{\partial y_j} \frac{\partial \tau_N}{\partial \eta_j} - \frac{\partial \nu_i}{\partial \eta _j} \frac{\partial \tau_N}{\partial y_j} \right) + \grandO_{N+6}= \grandO_{N+2}.$$
Hence, 
$$\frac{i}{h} \ad_{N_2}(\tau_N) = \sum_{j=1}^k \nu_j(w) \frac{i}{h} \ad_{\vert \tilde{v}_j \vert^2}( \tau_N) + \grandO_{N+2},$$
and equation (\ref{eq1442}) becomes:
\begin{align}\label{eq1442bis}
R_N = K_N + \sum_{j=1}^k \nu_j(w) \frac{i}{h}\ad_{\vert \tilde{v}_j \vert^2}(\tau_N) + \grandO_{N+1}.
\end{align}
Moreover, $\frac{i}{h}\ad_{\vert \tilde{v}_j \vert^2}$ acts as:
$$\sum_{j=1}^k \nu_j(w) \frac{i}{h}\ad_{\vert \tilde{v}_j \vert^2}(v^{\alpha_1} \bar{v}^{\alpha_2} h^{\ell}) =  \langle \nu(w) , \alpha_2- \alpha_1 \rangle v^{\alpha_1} \bar{v}^{\alpha_2} h^{\ell}.$$

The definition of $r_{\mathbf{2}}$ ensures that $\langle \nu(w) , \alpha_2 - \alpha_1 \rangle$ does not vanish on a neighborhood of $w=0$ if $N = \vert \alpha_1 \vert + \vert \alpha_2 \vert + 2 \ell < r_{\mathbf{2}}$ and $\alpha_1 \neq \alpha_2$. Hence we can decompose every $R_N$ as in (\ref{eq1442bis}), where $K_N$ contains the terms with $\alpha_1 = \alpha_2$. These terms are exactly the ones commuting with $\vert \tilde{v}_j \vert^2$ for $1 \leq j \leq k$.
\end{proof}

\subsection{Quantized second normal form}\label{Chap4-sec4-3}

Now we can quantize Lemma \ref{Chap4-Lemma-Red-Nh0} and Lemma \ref{Chap4-formalBNF2} to prove the following Theorem.

\begin{theorem}\label{Chap4-Thm-2ndFormeNormale}
There are :\\
\begin{enumerate}
\item[(1)] A unitary operator $\UhII : \Ld( \R^{s+k}_{(y,t)}) \rightarrow \Ld(\R^{s+k}_{(y,t)})$ quantifying a symplectomorphism $\tilde{\Phi}_{\mathbf{2}} = \Phi_{\mathbf{2}} + \grandO((t,\tau)^2)$ microlocally near $0$,
\item[(2)] A function $f_{\mathbf{2}}^{\star} : \R^{2s}_{w} \times \R^{k}_J\times [0,1) \rightarrow \R$ which is $\Cinf$ with compact support, such that
$$\vert f_{\mathbf{2}}^{\star}(w,J_1,\cdots, J_k,\sqrt{\hbar}) \vert \leq C ( \vert J \vert + \sqrt{\hbar})^2,$$
\item[(3)] A $\sqrt{\hbar}$-pseudodifferential operator $\RhII$ with symbol $\grandO((t,\hbar^{-1/2}\tau,\hbar^{1/4})^{r_{\mathbf{2}}})$ on a neighborhood of $0$,
\end{enumerate}
such that
$$\UhII^* \Nh^{[1]} \UhII = \hbar \Mh + \hbar \RhII,$$
where $\Mh$ is the following $\hbar$-pseudodifferential operator :$$\Mh = \Oph \widehat{b}(w,s(w)) + \sum_{j=1}^k \Jh^{(j)} \Oph \nu_j + \Oph f^{\star}_{\mathbf{2}}(w,\Jh^{(1)}, \cdots, \Jh^{(k)}, \sqrt{\hbar}).$$
\end{theorem}

\begin{proof}
Lemma \ref{Chap4-Lemma-Red-Nh0} gives a symplectomorphism $\Phi_{\mathbf{2}}$ such that: 
\begin{align*}
N^{[1]}_{\hbar} \circ \Phi_{\mathbf{2}} &=  \hbar \widehat{b} (w,s(w))  + \sum_{j=1}^k \nu_j(w) \left( \tau_j^2 + \hbar t_j^2 \right) + \grandO(\vert t \vert^3 \vert \tau \vert^2) \\ & \quad + \grandO(\vert t \vert^3 \hbar) + \grandO(\hbar^2) + \grandO(\hbar \vert \tau \vert ) + \grandO(\vert \tau \vert^3) + \grandO( \vert t \vert \vert \tau \vert^2).
\end{align*}
We can apply the Egorov Theorem to get a Fourier Integral Operator $\VhII$ such that:
$$\VhII^* \Oph ( N^{[1]}_{\hbar} ) \VhII = \Oph ( \hat{N}_{\hbar} ),$$
with $\hat{N}_{\hbar} = N^{[1]}_{\hbar} \circ \Phi_{\mathbf{2}} + \grandO(\hbar^2).$ We define $$\mathsf{N}_h(y,\eta,t,\taut) = \hat{N}_{\hbar}(y,\eta,t,h\taut),$$
and following the notations of Section \ref{Chap4-sectionSecondFormalBNF}, we have the associated formal series:
$$\frac{1}{h^2}\mathsf{N}_h = N_0 + N_2 + \gamma, \quad \gamma \in \grandO_3.$$
We apply Lemma \ref{Chap4-formalBNF2} and we get formal series $\kappa, \rho$ such that:
\[ e^{\frac i h \ad_\rho} (N_0 + N_2 + \gamma) = N_0 + N_2 + \kappa + \grandO_{\rdeux}\,. \]
We take a compactly supported symbol $\mathsf{a}(h,w,t,\tilde{\tau})$ with Taylor series $\rho$. Then the operator
\begin{equation}\label{eq.eh.N.eh}
e^{ih^{-1} \Opt(\mathsf{a})} \Opt  ( h^{-2} \mathsf{N}_h ) e^{-ih^{-1} \Opt(\mathsf{a})}
\end{equation}
has a symbol with Taylor series $N_0 + N_2 + \kappa + \grandO_{\rdeux}$. Since $\kappa \in \grandO_3$ commutes with $\vert \tilde{v}_j \vert^2$, it can be written
$$\kappa = \sum_{2\vert \alpha \vert + 2\ell \geq 3} c^{\star}_{\alpha \ell}(w) \left( \vert \tilde{v}_1 \vert ^2 \right) ^{\star \alpha_1} ... \left( \vert \tilde{v}_k \vert^2 \right) ^{\star \alpha_k }  h^{\ell}.$$
If we take $f_{\mathbf{2}}^{\star}(h,w,J_1,...,J_k)$ a smooth compactly supported function with Taylor series:
$$[f_{\mathbf{2}}^{\star}] = \sum_{2\vert \alpha \vert + 2\ell \geq 3} c_{\alpha \ell}^{\star}(w) J_1^{\alpha_1}... J_k^{\alpha_k} h^{\ell},$$
then the operator \eqref{eq.eh.N.eh} is equal to
\[ \Opt N_0 + \Opt N_2 + \Oph f^{\star}_{\mathbf{2}}(h, w, \Jh^{(1)},...,\Jh^{(k)}) \]
modulo $\grandO_{\rdeux}$.

\medskip

Multiplying by $h^2$, and getting back to the $\hbar$-quantization, we get:
$$e^{ih^{-1} \Opt(\mathsf{a})} \Op (\hat{N}_{\hbar}) e^{-ih^{-1} \Opt(\mathsf{a})} = \hbar \Mh + \hbar \mathcal{R}_{\hbar},$$
with: $$\Mh = \Op \widehat{b}(w,s(w)) + \sum_{j=1}^k \Op \nu_j(w) \Jh^{(j)} + \Op f_{\mathbf{2}}^{\star}(\sqrt{\hbar}, w, \Jh^{(1)},...,\Jh^{(k)}),$$
and $\mathcal{R}_{\hbar}$ a $\sqrt{\hbar}$-pseudodifferential operator with symbol $\grandO_{\rdeux}$.
Note that $\Mh$ is a \newline $\hbar$-pseudodifferential operator whose symbol admits an expansion in powers of $\sqrt{\hbar}$.
\end{proof}

\subsection{Proof of Theorem \ref{Chap4-Thm-spec-Nh0}}\label{Chap4-sec4-4}

In order to prove Theorem \ref{Chap4-Thm-spec-Nh0}, we need the following microlocalization lemma.

\begin{lemma}\label{Chap4-lemma-microlocMh}
Let $\delta \in (0,1/2)$ and $c>0$. Let $\chi_0 \in \Cinf_0(\R^{2s}_{(y,\eta)})$ and $\chi_1 \in \Cinf_0(\R^{2k}_{(t,\tilde{\tau})})$ both equal to $1$ on a neighborhood of $0$. Then every eigenfunction $\psi_{\hbar}$ of $\Nh$ or $\hbar \Mh$ associated to an eigenvalue $\lambda_{\hbar} \leq \hbar (b_0 + c \hbar^{\delta})$ satisfies :
$$\psi_{\hbar} = \mathsf{Op}_{\sqrt{\hbar}}^w  \chi_0(\sqrt{\hbar}^{-\delta} (t,\tilde{\tau})) \Oph \chi_1( y,\eta) \psi_{\hbar} + \grandO(\hbar^{\infty}) \psi_{\hbar}.$$
\end{lemma}

\begin{proof}
Using our mixed quantization and $h= \sqrt{\hbar}$, we have $\Nh^{[1]} = \Opt \mathsf{N}_h^{[1]}$ with
$$\mathsf{N}_h^{[1]}(y,\eta,t,\tilde{\tau}) = h^2 \langle M(y,\eta,t) \tilde{\tau}, \tilde{\tau} \rangle + h^2 \widehat{b}(w,t) + f_{\mathbf{1}}^{\star}(y,\eta,t,h \tilde{\tau}, h^2).$$
The principal part of $\mathsf{N}_h^{[1]}$ is of order $h^2$, and implies a microlocalization of the eigenfunctions where
$$h^2 \langle M(w,t) \tilde{\tau}, \tilde{\tau} \rangle + h^2 \widehat{b}(w,t) \leq \lambda_h \leq h^2 (b_0 + c h^{2 \delta}).$$
Since $\widehat{b}$ admits a unique and non-degenerate minimum $b_0$ at $0$, this implies that $w$ lies in an arbitrarily small neighborhood of $0$, and that :
$$\vert t \vert^2 \leq C h^{2 \delta}, \quad  \vert \tilde{\tau} \vert^2 \leq C h^{2 \delta}.$$
The technical details follow the same ideas of Lemmas \ref{Chap4-Lemme-MicrolocLh}, \ref{Chap4-Lemme-MicrolocNh1} and \ref{Chap4-Lemme-MicrolocNh2}. Now we can focus on $\Mh$ whose principal symbol with respect to the $\Opt$-quantization is
$$\mathsf{M}_0(y,\eta,t,\tilde{\tau}) = \widehat{b}(y,\eta,s(y,\eta)) + \sum_{j=1}^k \nu_j(y,\eta) ( \tilde{\tau}_j^2 + t_j^2).$$
Hence its eigenfunctions are microlocalized where
$$\widehat{b}(y,\eta,s(y,\eta)) + \sum_{j=1}^k \nu_j(y,\eta) ( \tilde{\tau}_j^2 + t_j^2) \leq b_0 + c h^{2\delta},$$
which implies again that $w$ lies in an arbitrarily small neighborhood of $0$ and that
$$\vert t \vert^2 \leq C h^{2\delta}, \quad \vert \tilde{\tau} \vert^2 \leq C h^{2\delta}.$$
\end{proof}

Using the same method as before, we deduce from Theorem \ref{Chap4-Thm-2ndFormeNormale} and Lemma \ref{Chap4-lemma-microlocMh} a comparison of the spectra of $\Nh^{[1]}$ and $\Mh$. With the notation $$N_{\hbar}^{\mathsf{max}}(c,\delta) = \max \lbrace n \in \N, \quad \lambda_n \left( \Nh^{[1]} \right) \leq \hbar (b_0 + c \hbar^{\delta} ) \rbrace,$$ the following lemma concludes the proof of Theorem \ref{Chap4-Thm-spec-Nh0}.

\begin{lemma}
Let $\delta \in (0,1/2)$ and $c>0$. We have
$$\lambda_n(\Nh^{[1]}) = \hbar \lambda_n(\Mh) + \grandO( \hbar^{1+ \delta r_{\mathbf{2}} /2}),$$
uniformly with respect to $n \in [1,N_{\hbar}^{\mathsf{max}}(c,\delta)]$.
\end{lemma}

\begin{proof}
We use the same method as before (see lemma \ref{Chap4-lemme-ineq1}). The remainder $\RhII$ is $\grandO((t,\tilde{\tau},\sqrt{\hbar})^{r_{\mathbf{2}}})$ and the eigenfunctions are microlocalized where  $\vert t \vert + \vert \tilde{\tau} \vert \leq C \hbar^{\delta/2}$. Hence the $\hbar \RhII$ term yield an error in $\hbar^{1 + \delta r_{\mathbf{2}} /2}$.
\end{proof}

\section{Proof of Corollary \ref{cor.reduction.Mh0}}\label{sec.Coro1}

In this section we prove that the spectrum of $\Lh$ below $\hbar b_0 + \hbar^{3/2}(\nu(0)+2c)$ is given by the spectrum of $\hbar \Mh^{[1]}$, up to $\grandO(\hbar^{r/4-\varepsilon})$. We recall that $c \in (0, \min_j \nu_j(0))$ and $r = \min (2 \run, \rdeux+4)$.

\medskip

We can apply Theorem \ref{thm.main.first.bnf}, for $b_1 > b_0$ arbitrarily close to $b_0$. Thus the spectrum of $\Lh$ in $(-\infty, b_1 \hbar)$ is given by the spectrum of $\bigoplus_{n\in \N^s} \Nh^{[n]}$ modulo $\grandO(\hbar^{\run/2 - \varepsilon}) = \grandO(\hbar^{r/4-\varepsilon})$. Moreover, the symbol of $\Nh^{[n]}$ for $n \neq (1, \cdots , 1)$ satisfies
\[ N_\hbar^{[n]}(y,\eta,t,\tau) \geq \hbar (b_0 + 2 \min \beta_j - C \hbar) \,, \]
et we deduce from the G\aa rding inequality that
\[ \langle \Nh^{[n]} \psi, \psi \rangle \geq \hbar b_1 \Vert \psi \Vert^2 \,, \quad \forall \psi \in \Ld(\R^{s+k}) \,, \]
if $b_1$ is close enough to $b_0$. Hence the spectrum of $\Lh$ below $b_1 \hbar$ is given by the spectrum of $\Nh^{[1]}$. Then, we apply Theorem \ref{Chap4-Thm-spec-Nh0} for $\delta$ close enough to $1/2$, and we see that the spectrum of $\Nh^{[1]}$ below $(b_0 + \hbar^\delta)\hbar$ is given by the spectrum of $\bigoplus_{n \in \N^k} \hbar \Mh^{[n]}$, modulo $\grandO(\hbar^{1+ \rdeux/4 - \varepsilon}) = \grandO(\hbar^{r/4})$. The symbol of $\Mh^{[n]}$ for $n \neq 1$ satisfies:
\[ \Mh^{[n]}(y,\eta) \geq b_0 + \hbar^{1/2} \sum_{j=1}^k \nu_j(y,\eta) (2n_j-1) - C \hbar \,,\]
and the eigenfunctions of $\Mh^{[n]}$ are microlocalized on an arbitrarily small neighborhood of $(y,\eta)=0$ (Lemma \ref{Chap4-lemma-microlocMh}), and $M_\hbar^{[n]}$ satisfies on this neighborhood:
\begin{align*}
M_\hbar^{[n]}(y,\eta) &\geq b_0 + \hbar^{1/2} \sum_{j=1}^k \nu_j(0) (2n_j-1) - \hbar^{1/2} \varepsilon - C \hbar \\
& \geq b_0 + \hbar^{1/2} (\nu(0) + 2 \min_j \nu_j(0) - \varepsilon) -C \hbar\,.
\end{align*}
Using the G\aa rding inequality, the spectrum of $\Mh^{[n]}$ ($n \neq 1$) is thus $\geq b_0 + \hbar^{1/2}(\nu(0)+2c)$ for $\varepsilon$ and $\hbar$ small enough. It follows that the spectrum of $\Nh^{[1]}$ below $\hbar b_0 + \hbar^{3/2}(\nu(0)+2c)$ is given by the spectrum of $\hbar \Mh^{[1]}$.

\section{Proof of Corollary \ref{cor.eigenvalue.asymptotics}}\label{sec.Coro2}

In this section we explain where the asymptotics for $\lambda_j(\Lh)$ come from (Corollary \ref{cor.eigenvalue.asymptotics}). First we can use Corollary \ref{cor.reduction.Mh0} so that the spectrum of $\Lh$ below $\hbar b_0 + \hbar^{3/2}(\nu(0)+2c)$ is given by $\Mh^{[1]}$, modulo $\grandO(\hbar^{r/4-\varepsilon})$. But the symbol of $\Mh^{[1]}$ has the following expansion
\[ M_{\hbar}^{[1]}(w) = \widehat{b}(w, s(w)) + \hbar^{1/2} \nu(0) + \hbar^{1/2} \nabla \nu (0) \cdot w + \hbar \tilde{c}_0 + \grandO( \hbar w +\hbar^{3/2} + \hbar^{1/2} w^2 )\,, \]
with $\nu(w) = \sum_{j=1}^k$. The principal part admits a unique minimum at $0$, which is non-degenerate. The asymptotics of the first eigenvalues of such an operator are well-known. First one can make a linear change of canonical coordinates diagonalizing the Hessian of $\hat b$, and get a symbol of the form
\[ \widehat{M}_{\hbar}^{[1]}(w) = b_0 + \sum_{j=1}^s \mu_j \left( \eta_j^2 + y_j^2 \right) + \hbar^{1/2} \nu(0) + \hbar^{1/2} \nabla \nu (0) \cdot w + \hbar \tilde{c}_0 + \grandO( w^3 + \hbar w  + \hbar^{3/2} + \hbar^{1/2} w^2) \,. \]
One can factor the $\nabla \nu (0) \cdot w$ term to get
\begin{align*}
\widehat{M}_{\hbar}^{[1]}(w) = b_0 + \sum_{j=1}^s \mu_j & \left(  \left(\eta_j + \frac{\partial_{\eta_j} \nu (0)}{2\mu_j} \hbar^{1/2} \right)^2 + \left( y_j + \frac{\partial_{y_j} \nu (0)}{2\mu_j} \hbar^{1/2} \right)^2 \right) + \hbar^{1/2} \nu(0) + \hbar c_0\\ 
&+ \grandO( w^3 + \hbar w + \hbar^{3/2} + \hbar^{1/2} w^2 )\,,
\end{align*}
with a new $c_0 \in \R$. Conjugating $\Op \widehat{M}_\hbar^{[1]}$ by the unitary operator $U_\hbar$,
\[ U_{\hbar} v(x) = \mathsf{exp} \left( \frac{i}{\sqrt{\hbar}} \sum_{j=1}^s \frac{\partial_{\eta_j} \nu (0)}{2 \mu_j} y_j \right)  v \left( x - \sum_{j=1}^s \frac{\partial_{y_j} \nu (0)}{2 \mu_j} \hbar^{1/2} \right)\,, \]
amounts to make a phase-space translation and reduces the symbol to
\begin{align*}
\tilde{M}_{\hbar}^{[1]}(w) = b_0 + \sum_{j=1}^s \mu_j (\eta_j^2 + y_j^2) + \hbar^{1/2} \nu(0) + \hbar c_0 + \grandO( w^3 + \hbar w + \hbar^{3/2} + \hbar^{1/2} w^2).
\end{align*}
For an operator with such symbol (i.e. harmonic oscillator + remainders) one can apply the results of \cite{CharlesVuNgoc} (Thm 4.7) or \cite{HelSjo1}, and deduce that the $j$-th eigenvalue $\lambda_j(\Mh^{[1]})$ admits an asymptotic expansion in powers of $\hbar^{1/2}$ such that
\[ \lambda_j(\Mh^{[1]}) = b_0 + \hbar^{1/2} \nu(0) + \hbar( c_0 + E_j) + \hbar^{3/2} \sum_{m=0}^{\infty} \alpha_{j,m} \hbar^{m/2}\,, \]
where $\hbar E_j$ is the $j$-th repeated eigenvalue of the harmonic oscillator with symbol \newline $\sum_{j=1}^s \mu_j (\eta_j^2 + y_j^2)$.

\bigskip

\appendix

\section{Local coordinates} \label{sectioncoordinates}

If we choose local coordinates $q=(q_1,...,q_d)$ on $M$, we get the corresponding vector fields basis $(\partial_{q_1}, ..., \partial_{q_d})$ on $T_qM$, and the dual basis $(\dd q_1 , ..., \dd q_d)$ on $T_qM^*$. In these basis, $g_q$ can be identified with a symmetric matrix $(g_{ij}(q))$ with determinant $\vert g \vert$, and $g^*_q$ is associated with the inverse matrix $(g^{ij}(q))$. We can write the $1$-form $A$ and the $2$-form $B$ in the coordinates:
$$A \equiv A_1 \dd q_1 + ... + A_d \dd q_d,$$
$$B= \sum_{i<j} B_{ij} \dd q_i \wedge \dd q_j,$$
with $\A = (A_j)_{1 \leq j \leq d} \in \mathcal{C}^{\infty}(\R^d,\R^d)$ and
\begin{align}\label{BenCoordonnees}
B_{ij} = \partial_i A_j - \partial_j A_i = (\ ^t \dd \A - \dd \A)_{ij}.
\end{align}
Let us denote by $(\B_{ij}(q))_{1 \leq i,j \leq d}$ the matrix of the operator $\B(q) : T_qM \rightarrow T_qM$ in the basis $(\partial_{q_1} , ..., \partial_{q_d})$. With this notation, equation (\ref{defBmatrix}) relating $\B$ to $B$ can be rewritten:
$$\forall Q,\tilde{Q} \in \R^d, \quad \sum_{ijk} g_{kj} \B_{ki} Q_i \tilde{Q}_{j} = \sum_{ij} B_{ij} Q_i \tilde{Q}_j,$$
which means that
\begin{align} \label{relationBBmatrice}
\forall i,j, \quad B_{ij} = \sum_k g_{kj} \B_{ki}.
\end{align}
Finally, in the coordinates, $H$ is given by:
\begin{align}\label{HamiltonienCoordonnees}
H(q,p) = \sum_{i,j} g^{ij}(q) (p_i - A_i(q))(p_j - A_j(q)),
\end{align}
and $\Lh$ acts as the differential operator:
\begin{align}\label{Lhcoord}
\Lh^{\text{coord}} = \sum_{k,l =1}^d \vert g \vert^{-1/2} (i\hbar \partial_k + A_k) g^{kl} \vert g \vert^{1/2} (i\hbar \partial_l+A_l).
\end{align}

\section{\textbf{Darboux-Weinstein lemmas}}

We used the following presymplectic Darboux Lemma.

\begin{theorem}
Let $M$ be a $d$-dimensional manifold endowed with a closed constant-rank two form $\omega$. We denote $2s$ the rank of $\omega$ and $k$ the dimension of its kernel. For every $q_0 \in M$, there are a neighborhood $V$ of $q_0$, a neighborhood $U$ of $0 \in \R^{2s+k}_{(y,\eta,t)}$, and a diffeomorphism:
$$\varphi : U \rightarrow V,$$
such that $$\varphi^* \omega = \dd \eta \wedge \dd y.$$
\end{theorem}

We also used the following Weinstein result (see \cite{Weinstein}). We follow the proof given in \cite{Birkhoff2D}.

\begin{theorem}\label{thm.Weinstein}
Let $\omega_0$ and $\omega_1$ be two $2$-forms on $\R^{d}$ which are closed and non-degenerate. Let us split $\R^d$ into $\R^k_x \times \R^{d-k}_{y}$. We assume that $\omega_0 = \omega_1 + \grandO( \vert x \vert^\alpha )$, for some $\alpha \geq 1$. Then there exists a neighborhood of $0 \in \R^d$ and a change of coordinates $\psi$ on this neighborhood such that:
$$\psi^* \omega_1 = \omega_0 \quad \text{and} \quad \psi = \mathsf{Id}+ \grandO(\vert x \vert^{\alpha+1}).$$
\end{theorem}

\begin{proof}
First we recall how to find a $1$-form $\sigma$ on a neighborhood of $x=0$ such that:
$$\tau := \omega_1 - \omega_0 = \dd \sigma, \quad \text{and} \quad \sigma = \grandO(\vert x \vert^{\alpha+1}).$$
We define the family $(\phi_t)_{0 \leq t\leq 1}$ by:
$$\phi_t(x,y) = (tx,y).$$
We have:
\begin{align}\label{eq15291}
\phi_0^* \tau = 0 \quad \text{and} \quad \phi_1^* \tau = \tau.
\end{align}
Let us denote by $X_t$ the vector field associated with $\phi_t$:
$$X_t = \frac{\dd \phi_t}{\dd t} \circ \phi_t^{-1} = t^{-1}(x,0).$$
The Lie derivative of $\tau$ along $X_t$ is given by $\phi_t^* \mathcal{L}_{X_t} \tau = \frac{\dd}{\dd t} \phi_t^* \tau$. From the Cartan formula we have:
$$\mathcal{L}_{X_t} \tau = \iota(X_t) \dd \tau + \dd (\iota(X_t)).$$ Since $\tau$ is closed, $\dd \tau = 0$, and:
\begin{align}\label{eq15290}
\frac{\dd}{\dd t} \phi_t^* \tau = \dd (\phi_t^* \iota(X_t) \tau ).
\end{align}
We choose the following $1$-form (where $(e_j)$ denotes the canonical basis of $\R^d$):
$$\sigma_t := \phi_t^* \iota(X_t) \tau = \sum_{j=1}^k x_j \tau_{\phi_t(x,y)} (e_j, \nabla \phi_t(.)) = \grandO( \vert x \vert^{\alpha +1}).$$
Equation (\ref{eq15290}) shows that $t \mapsto \phi_t^* \tau$ is smooth on $[0,1]$. Thus, we can define $\sigma = \int_0^1 \sigma_t \dd t$. It follows from (\ref{eq15290}) and (\ref{eq15291}) that:
$$\frac{\dd}{\dd t} \phi_t^* \tau = \dd \sigma_t \quad \text{and} \quad \tau = \dd \sigma.$$
Then we use the Moser deformation argument. For $t\in [0,1]$, we let $\omega_t = \omega_0 + t (\omega_1 - \omega_0).$ The $2$-form $\omega_t$ is closed and non degenerate on a small neighborhood of $x=0$. We look for $\psi_t$ such that:
$$\psi_t^* \omega_t = \omega_0.$$
For that purpose, let us determine the associated vector field $Y_t$:
$$\frac{\dd}{\dd t}\psi_t = Y_t(\psi_t).$$
The Cartan formula yields:
$$0 = \frac{\dd}{\dd t} \psi_t^* \omega_t = \psi_t^* \left( \frac{\dd}{\dd t} \omega_t + \iota(Y_t)\dd \omega_t + \dd ( \iota(Y_t)\omega_t) \right).$$
So :
$$\omega_0 -\omega_1 = \dd ( \iota (Y_t) \omega_t ),$$
and we are led to solve:
$$\iota(Y_t)\omega_t = -\sigma.$$
By non degeneracy of $\omega_t$, this determines $Y_t$. $\psi_t$ exists until time $t=1$ on a small enough neighborhood of $x=0$, and $\psi_t^*\omega_t = \omega_0$. Thus $\psi = \psi_1$ is the desired diffeomorphism. Since $\sigma = \grandO(\vert x \vert^{\alpha+1})$, we get $\psi = \mathsf{Id} + \grandO(\vert x \vert^{\alpha+1}).$
\end{proof}

\section{Pseudodifferential operators}

We refer to \cite{Zworski} and \cite{Martinez} for the general theory of $\hbar$-pseudodifferential operators. If $m\in \Z$, we denote by $$S^m(\R^{2d})= \lbrace a \in \mathcal{C}^{\infty}(\R^{2d}), \vert \partial^{\alpha}_x \partial^{\beta}_{\xi} a \vert \leq C_{\alpha \beta} \langle \xi \rangle ^{m-\vert \beta \vert}, \quad \forall \alpha, \beta \in \N^d \rbrace$$ the class of Kohn-Nirenberg symbols. If $a$ depends on the semiclassical parameter $\hbar$, we require that the coefficients $C_{\alpha \beta}$ are uniform with respect to $\hbar \in (0, \hbar_0]$. For $a_{\hbar} \in S^m(\R^{2d})$, we define its associated Weyl quantization $\Op(a_{\hbar})$ by the oscillatory integral
$$\mathcal{A}_{\hbar}u(x) = \Op (a_{\hbar})u(x)= \frac{1}{(2\pi \hbar)^d} \int _{\R^{2d}} e^{\frac{i}{\hbar} \langle x-y, \xi \rangle } a_{\hbar}\left(\frac{x+y}{2},\xi \right) u(y)\dd y \dd \xi,$$ and we denote:
$$a_{\hbar} = \sigma_{\hbar}(\mathcal{A}_{\hbar}).$$

If $M$ is a compact manifold, a pseudodifferential operator $\mathcal{A}_{\hbar}$ on $\Ld(M)$ is an operator acting as a pseudodifferential operator in coordinates. Then the principal symbol of $\mathcal{A}_{\hbar}$  (and its Kohn-Nirenberg class) does not depend on the coordinates, and we denote it by $\sigma_0(\mathcal{A}_{\hbar}).$ The subprincipal symbol $\sigma_1(\mathcal{A}_{\hbar})$ is also well-defined, up to imposing the charts to be volume-preserving (in other words, if we see $\mathcal{A}_{\hbar}$ as acting on half-densities, its subprincipal symbol is well defined). In the case where $M$ is a compact manifold, $\Lh$ is a pseudodifferential operator, and its principal and subprincipal symbols are:
$$\sigma_0(\Lh) = H, \quad \sigma_1(\Lh) = 0.$$

If $M=\R^d$ and $m$ is an order function on $\R^{2d}$, we denote by
$$S(m) = \lbrace a \in \mathcal{C}^{\infty}(\R^{2d}), \vert \partial^{\alpha}_x \partial^{\beta}_{\xi} a \vert \leq C_{\alpha \beta} m(x,\xi), \quad \forall \alpha, \beta \in \N^d \rbrace$$
the class of standard symbols, and we similarly define the operator $\Op(a)$ for such symbols. In this case, we assume that $B$ belongs to some standard class. This is equivalent to assume that $H$ belongs to some (other) standard class. Then, $\Lh$ is a pseudodifferential operator with total symbol $H$.

\section{\textbf{Egorov Theorem}}

In this paper, we used several versions of the Egorov Theorem. See for example \cite{Robert}, \cite{Zworski}, or \cite{Birkhoff3D}.

\begin{theorem}
Let $P$ and $Q$ be $\hbar$-pseudodifferential operators on $\R^d$, with symbols $p \in S(m)$, $q\in S(m')$, where $m$ and $m'$ are order functions such that:
$$m' = \grandO(1), \quad mm' = \grandO'(1).$$
Then the operator $e^{\frac{i}{\hbar}Q} P e^{-\frac{i}{\hbar}Q}$ is a pseudodifferential operator whose symbol is in $S(m)$, and its symbol is:
$$p \circ \kappa + \hbar S(1),$$
where the canonical transformation $\kappa$ is the time-$1$ Hamiltonian flow associated with $q$.
\end{theorem}

We can use this result with the $\sqrt{\hbar}$-quantization, to get an Egorov Theorem for our mixed quantization $\Opt$.

\begin{theorem}
Let $P$ be a $h$-pseudodifferential operator on $\R^d$, and $\mathsf{a} \in \mathcal{C}^{\infty}_0(\R^{2d})$. Then
$$e^{\frac{i}{\hbar}\Opt(\mathsf{a})} P e^{-\frac{i}{\hbar}\Opt( \mathsf{a})}$$
is a $h$-pseudodifferential operator on $\R^d$.
\end{theorem}

\begin{proof}
$\Opt(\mathsf{a})$ is a $h$-pseudodifferential operator. Thus, we can apply the Egorov Theorem, and we deduce that 
$e^{\frac{i}{\hbar}\Opt(\mathsf{a})} P e^{-\frac{i}{\hbar}\Opt( \mathsf{a})}$ is a $h$-pseudodifferential operator on $\R^d$.
\end{proof}

\bibliographystyle{plain}
\bibliography{biblio-these.bib}

\end{document}